\documentclass{article}
\usepackage[utf8]{inputenc}

\title{Is `being above the median' a noise sensitive property?}
\author{Daniel Ahlberg\footnote{Department of Mathematics, Stockholm University.}\; and Daniel de la Riva\footnote{Department of Mathematics, University of British Columbia.}}

\date{\vspace{-5ex}}
\usepackage{amsmath}
\usepackage{color}
\usepackage{float}
\usepackage{scalerel}[2016/12/29]
\usepackage{amsthm}
 \usepackage{bbm}
 \usepackage{comment}
\usepackage{dirtytalk}
\usepackage{mathtools}

\DeclarePairedDelimiter\floor{\lfloor}{\rfloor}

\usepackage{enumerate}
\usepackage[maxbibnames=99]{biblatex}
\usepackage{amsfonts}
\usepackage{graphicx}
\usepackage{sectsty}
\usepackage{dsfont}
\usepackage{amsmath,amssymb,amsfonts}
\usepackage{tikz}
\usetikzlibrary{hobby}
\sectionfont{\large}
\graphicspath{ {images/} }
\newtheorem{theorem}{Theorem}

\newtheorem{proposition}[theorem]{Proposition}

\newtheorem{remark}[theorem]{Remark}
\newtheorem{lemma}[theorem]{Lemma}

\def\E{{\mathbb{E}}}

\def\P{{\mathbb{P}}}
\def\Var{\textup{Var}}
\def\Inf{\textup{Inf}}
\def\ind{\mathbbm{1}}
\def\eps{\varepsilon}
\numberwithin{theorem}{section}
\usepackage{biblatex}
\def\R{{\mathbb{R}}}
\def\Z{{\mathbb{Z}}}

\DeclareFieldFormat{pages}{#1}
\renewbibmacro{in:}{%
  \ifentrytype{article}
    {}
    {\bibstring{in}%
     \printunit{\intitlepunct}}}

\begin{document}

\maketitle

\begin{abstract}

Assign independent weights to the edges of the square lattice, from the uniform distribution on $\{a,b\}$ for some $0<a<b<\infty$. The weighted graph induces a random metric on $\Z^2$. Let $T_n$ denote the distance between $(0,0)$ and $(n,0)$ in this metric. The distribution of $T_n$ has a well-defined median. Itai Benjamini asked in 2011 if the sequence of Boolean functions encoding whether $T_n$ exceeds its median is noise sensitive?
In this paper we present the first progress on Benjamini's problem.
More precisely, we study the minimal weight along any path crossing an $n\times n$-square horizontally and whose vertical fluctuation is smaller than $n^{1/22}$, and show that for this observable, `being above the median' is a noise sensitive property.

{\em Keywords:} First-passage percolation; noise sensitivity; moderate deviations.

{\em Funding:} Research in part supported by the Swedish Research Council through grant 2021-03964.
\end{abstract}

\section{Introduction}\label{Introduction}

Central in statistical physics is the notion of a phase transition, i.e.\ a sudden change of behaviour as some parameter of the model is changed. As a consequence, configurations that correlate well on a microscopic scale may look radically different on a macroscopic scale, if they correspond to different sides of the transition. However, it is also possible for highly correlated configurations to behave differently, despite having the same law.
A formal framework, in the context of Boolean functions, in which questions like this could be studied was introduced in a seminal paper by Benjamini, Kalai and Schramm~\cite{BKS99}. Let $\omega\in\{0,1\}^n$ be chosen uniformly at random, and obtain $\omega^\eps$ from $\omega$ by independently resampling the coordinates with probability $\eps\in(0,1)$. A sequence $(f_n)_{n\ge1}$ of Boolean functions $f_n:\{0,1\}^n\to\{0,1\}$ is said to be {\bf noise sensitive} if for every $\eps>0$
\begin{equation}\label{eq:ns_def}
\E\big[f_n(\omega)f_n(\omega^\eps)\big]-\E[f_n(\omega)]^2\to0\quad\text{as }n\to\infty.
\end{equation}

In~\cite{BKS99}, the authors gave the first example of noise sensitivity, in particular establishing noise sensitivity in planar Bernoulli percolation at criticality. In order to do this, they developed methods by which noise sensitivity could be established, that remain relevant to this day.
Later works have established analogous results for percolation in the continuum, based on Poisson~\cite{ABGM14,QuenchedVoronoi,LPY} and Gaussian processes~\cite{BargmannFock}, as well as in the context of random graphs~\cite{LS15}. For these models, central observables have a Boolean outcome. For many other models in the realm of random spatial processes, the main observables are not Boolean, but real-valued functions on the space of configurations. This is the case for a variety of disordered systems, polymer models, and spatial growth models such as first- and last-passage percolation.

In September 2011, at the doctoral defense of the first author, Itai Benjamini proposed a natural approach to explore the sensitivity to small perturbations of real-valued observables. The approach can be synthesised concisely with the words: \emph{Is `being above the median' a noise sensitive property?} The purpose of this paper is to present the first progress on Benjamini's problem.

\subsection{Model description and main result}

In the most well-studied setting, first-passage percolation is the study of the random metric space that arise by assigning non-negative independent random weights, from some common distribution $F$, to the edges of the $\Z^2$ nearest-neighbour lattice. For simplicity, we shall in this paper stick to the planar setting, and we shall for most of the paper assume that $F$ is uniformly distributed on $\{a,b\}$, for some $0<a<b<\infty$. The edge weights induce a metric $T$ on $\Z^2$ as follows: For $u,v\in\Z^2$, set
$$
T(u,v):=\inf\big\{T(\Gamma):\Gamma\text{ is a path from $u$ to $v$}\big\},\quad\text{where}\quad T(\Gamma):=\sum_{e\in\Gamma}\omega_e.
$$
The infimum in $T(u,v)$ is known to be attained for some finite path, although this path does not have to be unique. We let $\pi(u,v)$ denote this path, and apply some deterministic rule for selecting one in case it is not unique.

For $n\ge1$, set $T_n:=T(0,n{\bf e}_1)$ and $\pi_n:=\pi(0,n{\bf e}_1)$ for brevity, where ${\bf e}_1:=(1,0)$ denote the first coordinate vector. A standard consequence of the Subadditive Ergodic Theorem is the existence of a constant $\mu$, known as the {\bf time constant}, such that almost surely
\begin{equation}\label{eq:timeconst}
    \frac{T_n}{n}\to\mu\quad\text{as }n\to\infty.
\end{equation}
Also, $\pi_n$ is known to be of linear order, although it is not known to have a well-defined asymptotic speed. In approaching the problem of the current paper, one soon requests a finer description of the order of fluctuations, both for $T_n$ around its mean, and $\pi_n$ away from the coordinate axis. Predictions from the physics literature~\cite{karparzha86}, which have been established for related models of last-passage percolation~\cite{BDJ99,J00}, suggest that fluctuations of $T_n$ are order $n^{1/3}$ and transversal fluctuations of $\pi_n$ are order $n^{2/3}$. 

The approach proposed by Itai Benjamini, in September 2011, to explore questions of noise sensitivity in the context of first-passage percolation (and for other real-valued observables) can be described as follows: The distance $T_n$ is a random variable, whose distribution may be unknown. This distribution has a median $m_n$, and for large $n$, this median can be expected to split the distribution of $T_n$ roughly in half, i.e.\ that
\begin{equation}\label{eq:balanced}
\P(T_n<m_n)\approx\P(T_n>m_n)\approx\frac12.
\end{equation}
Under the assumption that the weight distribution $F$ is supported on $\{a,b\}$, for some $0<a<b<\infty$, the event that $T_n$ exceeds its median can be encoded as a Boolean function. If~\eqref{eq:balanced} holds, then this function is non-degenerate. It is thus possible, as proposed by Benjamini, to investigate whether $T_n$ exceeding its median is a noise sensitive property, within the framework of Boolean functions.

In this paper we shall present the first progress on Benjamini's problem. We have not been able to answer the question as it has been formulated above, for reasons that we shall elaborate upon below. As stated the question thus remains open. Indeed, although~\eqref{eq:balanced} trivially holds for continuous weight distributions, it seems to remain unknown whether, uniformly in $n$,
$$
\P(T_n<m_n)>0\quad\text{and}\quad\P(T_n>m_n)>0,
$$
for some median of $T_n$, when $F$ is discrete. See, however,~\cite{C19,DHHX20} for results in this direction.


In order to circumvent the difficulties faced above, we shall make two simplifications to Benjamini's problem. First, we replace point-to-point passage times by horizontal crossing times of squares, and hence increase the symmetry in the problem. Second, we restrict the transversal fluctuations allowed by paths crossing the squares.
Given $k\ge1$, let $\mathcal{P}_k(n)$ denote the set of nearest-neighbour paths contained the `square' $[0,n]\times[0,n-1]$ that connect the left side to the right, and whose vertical displacement is at most $k$, i.e.\ are contained in $[0,n]\times[m,m+k-1]$ for some $m$, and set
\begin{equation}\label{eq:tau}
\tau(n,k):=\inf\big\{T(\Gamma):\Gamma\in\mathcal{P}_k(n)\big\}.
\end{equation}

Our main result is the following theorem, which makes the first progress on Benjamini's problem. Our result is formulated for an arbitrary quantile of the crossing variable, and not just its median.


\begin{theorem}\label{thm:main}
    Suppose that $F$ is uniformly distributed on $\{a,b\}$ for some $0<a<b<\infty$.
    Let $\alpha<1/22$ and $\beta\in(0,1)$ be fixed. For any sequence $(k_n)_{n\ge1}$ such that $k_n\le n^\alpha$, and for any $\beta$-quantile $q_\beta=q_\beta(n)$ of $\tau(n,k_n)$, we have
    $$
    \lim_{n\to\infty}\P\big(\tau(n,k_n)<q_\beta\big)=1-\lim_{n\to\infty}\P\big(\tau(n,k_n)>q_\beta\big)=\beta,
    $$
    and the sequence $(f_n)_{n\ge1}$ of functions $f_n:=\ind_{\{\tau(n,k_n)>q_\beta\}}$ is noise sensitive.
\end{theorem}

\begin{remark}
In this paper we have chosen to work with weight distributions that are uniform on $\{a,b\}$, for some $0<a<b<\infty$. Extending our main result beyond distributions with two-valued support would require different techniques, as the method of influences used here is currently restricted to this setting. However, the argument extends to non-uniform distributions and dimensions $d\ge2$, although with a stronger restriction on the growth of $k_n$, and the additional condition that $F(a)<\vec{p}_c(d)$, where $\vec{p}_c(d)$ is the critical probability for oriented percolation on $\Z^d$. The latter condition is believed to be superfluous (as mentioned in the appendix), but required for the lower bound in Lemma~\ref{lma:CD_variance}.
Let us finally mention that the analogous result holds for the square replaced by an $n\times n$-torus, and $\tau(n,k_n)$ replaced by the minimal weight among all circuits crossing the torus horizontally, and whose transversal fluctuations are bounded by $k_n$.
Adapting the proof of Theorem~\ref{thm:main} is straightforward, and not presented here.
\end{remark}

Related to the study of noise sensitivity is the notion of `chaos', that stems from the physics literature on spin-glasses~\cite{bramoo87,fishus86}. In the context of first-passage percolation, chaos refers to the sensitivity of the distance-minimising path $\pi_n$ as opposed to the distance $T_n$. The first rigorous evidence of chaos was obtained by Chatterjee in two preprints~\cite{Chatterjee2008,Chatterjee2009DisorderCA}, later combined into a book~\cite{Chatterjee_2014}. That the first-passage metric is chaotic was established only recently, in work of Ahlberg, Deijfen and Sfragara~\cite{ADS1}. To state this result, let $\pi_n^\eps$ denote the distance-minimising path between the origin and $n{\bf e}_1$ with respect to the perturbed weights $\omega^\eps$. In the setting when $F$ is continuous and has finite moment of order $2+\log$, for instance, it was proved in~\cite{ADS1} that for $\eps=\eps(n)\gg\frac1n\Var(T_n)$ we have
$$
\E\big[|\pi_n\cap\pi_n^\eps|\big]=o(n).
$$
A significantly more detailed analysis was carried out by Ganguly and Hammond~\cite{Ganguly2020StabilityAC}, showing that the transition from stability to chaos indeed occurs at $\eps\asymp\frac1n\Var(T_n)$ for a certain integrable model of last-passage percolation. Their model is known to obey KPZ-scaling, which establishes a transition at $\eps\asymp n^{-1/3}$; see also~\cite{ADS2} for a related result.

Our result differ from the above in that it addresses the sensitivity of the metric $T$ as opposed to the structure minimising $T$, and (to our knowledge) this result is the first of its kind for a (supercritical) spatial growth model. Note, however, related work of Damron, Hanson, Harper and Lam~\cite{DHHL23} that establish the existence of exceptional times in a dynamical version of {\em critical} first-passage percolation.

\subsection{A tale of influences}

A key result from the original paper of Benjamini, Kalai and Schramm~\cite{BKS99} gives a criterion for a sequence of Boolean functions to be noise sensitive in terms of the notion of influences. The influence of bits is central in computer science, and has its origin in social choice theory. The {\bf influence} of a bit $i\in\{1,2,\ldots,n\}$ of a function $f:\{0,1\}^n\to\{0,1\}$ is defined as the probability, under the uniform measure on $\{0,1\}^n$, that the bit is decisive for the outcome of the function, i.e.\
\begin{equation}\label{eq:inf_def}
\Inf_i(f):=\P\big(f(\omega)\neq f(\sigma_i\omega)\big),
\end{equation}
where $\sigma_i:\{0,1\}^n\to\{0,1\}^n$ is the operator that flips the entry at position $i$.
The criterion, which has come to be known as the BKS Theorem, states that if
\begin{equation}\label{eq:inf_cond}
    \sum_{i=1}^n\Inf_i(f_n)^2\to0\quad\text{as }n\to\infty,
\end{equation}
then the sequence $(f_n)_{n\ge1}$ is noise sensitive.\footnote{The BKS Theorem was in~\cite{BKS99} proved for the uniform measure on $\{0,1\}^n$, but remains true for biased measures~\cite{ABGM14,KK13}.}

Apart from the computation of influences, as the BKS Theorem invites to, there are other methods by which noise sensitivity may be established. The main development has occurred with applications to Bernoulli percolation in mind: A method involving the revealment of randomised algorithms was developed by Schramm and Steif~\cite{SS10}; The Fourier spectrum of critical percolation was analysed by Garban, Pete and Schramm~\cite{GPS10}; A probabilistic approach was taken by Tassion and Vanneuville~\cite{TV}, inspired by Kesten's scaling relations. Neither of these routes seem easy to follow in our context.
Moreover, for monotone functions (which we are concerned with here) the criterion in~\eqref{eq:inf_cond} is both necessary and sufficient for a sequence to be noise sensitive. So, either directly or indirectly, verifying~\eqref{eq:inf_cond} is inevitable. This will, hence, be the route we take.

Let us start with a general observation. For functions $f:\{0,1\}^n\to\R$ that are Lipschitz, i.e.\ have bounded differences $|f(\omega)-f(\sigma_i\omega)|\le c$ for some constant $c>0$ and all $i$, if changing the value of a bit $i$ takes $f$ from below to above its median $m$ (or vice versa), then $f$ must have been within distance $c$ from the median. In particular, we have the distributional bound
\begin{equation}\label{eq:dist_bound}
\Inf_i(\ind_{\{f>m\}})\le\P\big(|f-m|\le c\big).
\end{equation}
Standard variance bounds for functions that are Lipschitz (i.e.\ having bounded differences) with constant $c$ give $\Var(f)\le c^2n$; see~\cite[Corollary~3.2]{BLM13}. Hence, one should expect $f$ to have fluctuations at scale no larger than $\sqrt{n}$, and the above distributional bound to give an upper bound on the influences which is of order $1/\sqrt{n}$ at best. This would amount to an upper bound on the sum of influences squared being a non-vanishing constant. This heuristic suggests that it cannot in general be sufficient to bound the influences simply considering the distribution of $f$.

Note that, regardless if we consider $T_n$ or $\tau(n,k)$, changing the value of an edge may affect the observable by at most $\pm(b-a)$, meaning that they are both Lipschitz. Hence, a simple distributional bound as in~\eqref{eq:dist_bound} will not suffice.
Using an observation from~\cite{BKS03}, we may link influential edges to edges on the geodesic. Recall that $\pi_n$ is the path (a path in case of multiple) attaining the infimum in $T_n$.
Then,
\begin{equation}\label{eq:inf_firstbound}
\Inf_e(\ind_{\{T_n>m_n\}})=2\,\P\big(\omega_e=a,e\text{ pivotal}\big)\le2\,\P\big(e\in\pi_n,|T_n-m_n|\le b-a\big),
\end{equation}
since if $e$ is pivotal\footnote{The edge $e$ is \emph{pivotal} if changing the value of $\omega_e$ changes the outcome of the (Boolean) function in question.} and $\omega_e=a$, then $e$ belongs to every geodesic for $T_n$, and hence $e\in\pi_n$. The predictions from KPZ universality suggest that $|T_n-m_n|\le b-a$ should occur with probability order $n^{-1/3}$, and that a typical edge being on the geodesic has probability order $n^{-2/3}$. For \emph{most} edges within distance $n^{2/3}$ or the coordinate axis the influence is thus order $1/n$, and for edges further away it is negligible. This amounts to a bound on the sum of influences squared that vanishes with $n$.

The above heuristic is merely conjectural, and we are nowhere close to establish statements like this in first-passage percolation. It is generally not even known whether
$$
\P\big(|T_n-m_n|\le b-a\big)\to0\quad\text{as }n\to\infty
$$
for some median $m_n$ of $T_n$. However, a result by Pemantle and Peres~\cite{PP94} implies such a statement for exponentially distributed edge weights.

In a first attempt to simplify Benjamini's problem it is tempting to replace the point-to-point passage times with the left-right crossing times of rectangles.
Let $\mathcal{P}(n,m)$ denote the set of nearest-neighbour paths contained in the rectangle $R(n,m):=[0,n]\times[0,m-1]$ that connect the left side to the right. We define the crossing time of the rectangle $R(n,m)$ as
\begin{equation}\label{eq:t_rect}
t(n,m):=\inf\big\{T(\Gamma):\Gamma\in\mathcal{P}(n,m)\big\}.
\end{equation}
In particular, we let $t_n:=t(n,n)$ denote the crossing time of the `square' $R(n,n)$. The increasing level of symmetry attained in this way is manifested in that all horizontal and all vertical bonds of the square an effect of $t_n$ that is roughly equal.\footnote{Admittedly, the square may have to be replaced by a torus, and the crossing time of the square by the circumference of the torus, for this to be fully rigorous.}
From the linear upper bound on the length of a geodesic due to Kesten~\cite{kesten_1980}, a bound analogous to~\eqref{eq:inf_firstbound} would give
$$
\Inf_e(\ind_{\{t_n>m_n\}})\le C\,\frac1n\,\P\big(|t_n-m_n|\le b-a\big),
$$
and hence
\begin{equation}\label{eq:infsq_bound}
\sum_e\Inf_e(\ind_{\{t_n>m_n\}})^2\le C\,\P\big(|t_n-m_n|\le b-a\big)^2.
\end{equation}
This shows that even if we spread out the contribution coming from `being on the geodesic', only considering the contribution from the geodesic, and not the distributional properties of the crossing time, will not suffice in order to deduce noise sensitivity.

Again, it remains unknown whether the probability in the right-hand side of~\eqref{eq:infsq_bound} vanishes as $n\to\infty$. In fact, also the weaker question whether the variance of $t_n$ diverges as $n\to\infty$ remains unknown; the best lower bound gives a constant. We refer the reader to the recent work of Damron, Houdr\'e and \"Ozdemir~\cite{damronfluctuation} for a further discussion in this direction.

\subsection{Distributional control over restricted paths}

We shall circumvent the above mentioned difficulties in calculating the influences by imposing a restriction on the transversal fluctuations of the paths admissible for crossing the square $R(n,n)$.

The restriction on transversal fluctuations does not result in a lower asymptotic velocity by which the square is crossed, as long as the allowed fluctuations diverge with $n$; see~\cite{ahl15,Chatterjee2009CentralLT}. That is, for any diverging sequence $(k_n)_{n\ge1}$ we have, almost surely,
$$
\mu=\lim_{n\to\infty}\frac{T_n}{n}=\lim_{n\to\infty}\frac{t_n}{n}=\lim_{n\to\infty}\frac{\tau(n,k_n)}{n}.
$$
However, it is expected that the restriction does have an effect on a lower order. For $k$ fixed, on the other hand, one may show that there exists $\mu_k>\mu$ such that almost surely
$$
\lim_{n\to\infty}\frac{\tau(n,k)}{n}=\mu_k.
$$

To see how the transversal restriction will help in calculating the influences, let us consider the case when $k=1$. Note that $\tau(n,1)$ is the sum of $n$ independent binomial random variables with parameters $n$ and $1/2$. Using moderate deviation estimates of the binomial distribution we are able to compute the asymptotic behaviour of the quantiles of their minimum as well as the influences. 
Note how the case $k=1$ is reminiscent of the classical Tribes function, introduced by Ben-Or and Linial~\cite{BL85}, but with polynomial-sized tribes as opposed to logarithmic. We treat the case $k=1$ in detail in Section~\ref{sectribes}, as special case of a larger family of polynomial Tribes functions.

For $k\ge2$ we may express $\tau(n,k)$ as the minimum of $n-k+1$ identically distributed, but dependent, variables as follows. Let $R_i(n,k)$ denote the rectangle $[0,n]\times[i,i+k-1]$, and $t_i(n,k)$ the horizontal crossing time of $R_i(n,k)$. Since every path in $\mathcal{P}_k(n)$ may fluctuate vertically at most $k$, it has to be contained in $R_i(n,k)$ for some $i=0,1,\ldots,n-k$. It follows that
$$
\tau(n,k)=\min_{i=0,1,\ldots,n-k}t_i(n,k).
$$

For fixed $k$, the distribution of these variables is asymptotically Gaussian, as proved (in parallel) by Ahlberg~\cite{ahl15} and Chatterjee and Dey~\cite{Chatterjee2009CentralLT}. In fact, the latter paper shows that the asymptotic normality continues to hold for $k=k(n)$ growing slower than $n^{1/3}$. The asymptotic normality will not be sufficient in itself, as we will need to peak into the tail of the distribution, in that $\tau(n,k)$ is a minimum of a large number of variables.
For that reason, we shall need to combine the approach from~\cite{Chatterjee2009CentralLT} with a Cram\'er-type moderate deviations theorem for triangular arrays (Theorem~\ref{ldthm2}), in order to obtain a moderate deviations theorem for first-passage percolation across thin rectangles (Theorem~\ref{thm:mdfpp}). With the moderate deviations estimates at hand, we will be able to approximate the asymptotic behaviour of quantiles and influences for the restricted crossing time $\tau(n,k)$, and prove Theorem~\ref{thm:main}.

We remark that the asymptotic normality is in itself not central to our approach. The relevant part is that it allows us to bound the influence of an edge by a rare enough event, whose probability we may compute.
As a by-product of our proof we obtain the following estimate on the fluctuations on $\tau(n,k)$.

\begin{theorem}\label{thm:fluctuations}
Suppose that $F$ is uniformly distributed on $\{a,b\}$ for some $0<a<b<\infty$.
For any $\alpha<1/22$ and any sequence $(k_n)_{n\ge1}$ such that $k_n\le n^\alpha$ we have
$$
\sup_{x\ge0}\P\big(\tau(n,k_n)\in[x,x+c]\big)=o\Big(\frac{1}{n^{1/22}}\Big).
$$
\end{theorem}

As a corollary, it follows from the above theorem that the probability that $\tau(n,k_n)$ is found in an interval of length $O(n^{1/22})$ is $o(1)$, and hence that for any $c<\infty$ and large enough $n\ge1$ we have $\Var\big(\tau(n,k_n)\big)\ge cn^{1/11}$ for any sequence $(k_n)_{n\ge1}$ satisfying the assumption of the theorem.

While we here focus on weight distributions supported on two points, we remark that our proof of Theorem~\ref{thm:fluctuations} goes through without change for weight distributions supported on $[a,b]$.
Apart from a result by Pemantle and Peres~\cite{PP94} for exponentially distributed edge weights, it remains unknown whether for every $c>0$ we have
\begin{equation}\label{eq:anticoncentration}
\sup_{x\ge0}\P\big(T_n\in[x,x+c]\big)\to0\quad\text{as }n\to\infty.
\end{equation}
It would be interesting to establish~\eqref{eq:anticoncentration} for a larger class of weight distributions.\footnote{Since the first appearance of this manuscript, Elboim~\cite{elboim} has established~\eqref{eq:anticoncentration} for weight distributions that are absolutely continuous with respect to Lebesgue measure. For discrete weight distributions it remains an open problem.}

The analogous problem for geodesics is the well-known `midpoint problem', which was posed by Benjamini, Kalai and Schramm in~\cite{BKS03}. Interestingly, this problem has been solved for continuous weight with finite mean by Ahlberg and Hoffman~\cite{AH}. Their result shows that for every edge $e$ we have
\begin{equation}\label{eq:midpoint}
\P\big(e\in\pi(-n{\bf e}_1,n{\bf e}_1)\big)\to0\quad\text{as }n\to\infty.
\end{equation}
Earlier work of Damron and Hanson~\cite{DH17} gave a conditional proof under plausible, but unverified, assumptions on the asymptotic shape. In more recent work, Dembin, Elboim and Peled~\cite{DEP} derive polynomial rates on the decay in~\eqref{eq:midpoint} for a more restrictive class of weight distributions. However, as mentioned above, without progress on the problem in~\eqref{eq:anticoncentration}, these results are insufficient for making further progress on Benjamini's problem.

\subsection{Organisation of the paper}

The rest of this paper is organised as follows. In Section~\ref{sectribes} we showcase out approach by considering a polynomial version of the classical Tribes function. In Section~\ref{sec:generalMD} we prove a Cram\'er-type moderate deviations theorem for triangular arrays. In Section~\ref{SecDistribution} we apply the moderate deviations theorem to prove a moderate deviations theorem for first-passage percolation across thin rectangles, which will allow us to analyse the asymptotic behaviour of the quantiles of our main observable. In Section~\ref{SecSideThm} we derive a preliminary version of our main theorem, in which we consider the minimum crossing time across disjoint rectangles. Finally, in Section~\ref{sec:main}, we prove our main results, and in Section~\ref{sec:conjectures} we elaborate upon some open problems.

\section{Polynomial Tribes}\label{sectribes}

In this section we illustrate our method in a simplified context. We shall prove that `being above the median' is a noise sensitive property for a class of functions that generalises the classical function known as Tribes, introduced in~\cite{BL85}. Since the classical Tribes function is a standard example in the analysis of Boolean functions, see e.g.~\cite{garban_steif_book}, and since our generalisation interpolates between the classical Tribes and the standard Majority function, we think that carrying out the analysis in some detail may be of independent interest. The reader should note that this section is independent of the rest of the paper, and can be thought of as a mere illustration of the general approach. However, there is a connection to first-passage percolation, as a special case of the class of functions introduced in this section coincides with $\tau(n,k)$ for $k=1$, i.e.\ when vertical displacement is not allowed, as explained below.

For every $\lambda \in (0,1)$ we partition $[n]:=\{1,2,\ldots,n\}$ into blocks of length $\ell_\lambda:=\floor{n^{\lambda}}$, and perhaps some leftover debris. We refer to each block as a \emph{tribe}. Given $\omega\in\{0,1\}^{[n]}$, sampled uniformly, we define $S_j$ as the sum of the coordinates of the $j$th tribe, for each of the $m_\lambda:=\floor{n/\ell_\lambda}$ tribes. Finally, let
\begin{equation}\label{GenTribes}
    S^{\lambda}:=\max_{1\leq j \leq m_\lambda} S_{j}
\end{equation}
denote the maximal number of 1s in any tribe. Note that we have suppressed the dependence on $n$ in the above notation. Note, moreover, that the choice $\lambda=1/2$ and $n=m^2$ corresponds to the weight of a left-right crossing of an $m\times m$-square among paths with no vertical displacement. The connection can be realised by interpreting $S^{1/2}$, for $n=m^2$, as the number of low-weight edges it is possible to use traversing an $m\times m$-square along a straight line, and hence $\tau(m,1)=aS^{1/2}+b(m-S^{1/2})$.

For each $\beta\in(0,1)$, let $q_{\lambda,\beta}$ denote any $\beta$-quantile of (the distribution of) $S^\lambda$. Define $f_{\lambda,\beta}$ to be the indicator function of the event $\{S^\lambda>q_{\lambda,\beta}\}$, i.e.\ that at least one tribe contains more than $q_{\lambda,\beta}$ 1s.

Naturally, our idea will be to study the behavior of $\P(f_{n,\beta}^{\lambda}=0)$, and if the sequence of functions $\{f_{n,\beta}^{\lambda}\}$ is noise sensitive. More precisely, we get the following result.

\begin{proposition}\label{maintribesthm}
For every $\lambda,\beta\in(0,1)$ we have, as $n\to\infty$, that $f_{\lambda,\beta}$ is noise sensitive and
$$
\P(f_{\lambda,\beta}=0)\to\beta.
$$
\end{proposition}

Due to the connection between the generalised tribes function and the left-right crossing of a square among paths of vertical displacement at most $k=1$ (i.e.\ no vertical displacement), the reader can note that by Proposition~\ref{maintribesthm}, for $n=m^2$, $\lambda=1/2$, $\beta\in (0,1)$ and any $\beta$-quantile $q_{\beta}$ of $\tau(m,1)$, gives us $\P(\tau(m,1)\le q_{\beta}) \rightarrow \beta,$ and that the indicator $\mathbbm{1}_{\{\tau(m,1)> q_{\beta}\}}$ is noise sensitive as $m\to\infty$, and hence proves a special case of Theorem~\ref{thm:main}.

We now move to the proof of Proposition~\ref{maintribesthm}. Since the number of 1s in a tribe follows a binomial distribution, and since our function asks for the maximal number of 1s in any tribe, we shall in the proof of the proposition make use of known estimates on the tail of the centred binomial. Let $X_n$ denote a binomially distributed random variable with parameters $n$ and $1/2$. The following estimates are a reformulation of Theorem~2 in Bahadur~\cite{bahadur60} (together with an observation in the summary); see also Theorem~1.13 in~\cite{simonmoderate}: For any sequence $x_n$ satisfying $1\ll x_n\ll n^{1/6}$ we have, as $n\to\infty$, that
\begin{align}
\label{eq:bin1}
    \P\big(X_n\ge n/2+x_n\sqrt{n}/2\big)&=(1+o(1))\frac{1}{x_n\sqrt{2\pi}}\exp\big(-x_n^2/2\big),\\
\label{eq:bin2}
    \P\big(X_n=\floor{n/2+x_n\sqrt{n}/2}\big)&=(1+o(1))\frac{\sqrt{2}}{\sqrt{\pi n}}\exp\big(-x_n^2/2\big).
\end{align}

We begin with a couple of lemmas determining the correct order of the $\beta$-quantiles of the generalised tribes function. For $\lambda,\beta \in (0,1)$ and $n\ge1$, let $s_{\lambda,\beta}=s_{\lambda,\beta}(n)$ be defined as
\begin{equation*}
s_{\lambda,\beta}:=\frac{\ell_\lambda}{2} + \frac{\sqrt{\ell_\lambda}}{2}\sqrt{2(1-\lambda)\log n -\log\log n - 2\log\Big(\sqrt{4\pi\big(1-\lambda\big)}\log \beta^{-1}\Big)}.
\end{equation*}

\begin{lemma}\label{case1conv}
For every $\lambda,\beta\in(0,1)$ we have $\P(S^{\lambda}\le s_{\lambda,\beta}) \to \beta$ as $n\to\infty$.
\end{lemma}


\begin{proof}
For $\{S^\lambda\le s_{\lambda,\beta}\}$ to occur, we need that all tribes to contain at most $s_{\lambda,\beta}$ 1s. Since tribes are disjoint it follows by independence that 
$$
\P\big(S^{\lambda}\le s_{\lambda,\beta}\big) = \Big(1 - \P\big(X_{\ell_\lambda}> s_{\lambda,\beta}\big)\Big)^{m_\lambda}.
$$
By~\eqref{eq:bin1}, and since by~\eqref{eq:bin2} the probability of attaining the value $\floor{s_{\lambda,\beta}}$ is of a lower order, we have that
$$
\P\big(X_{\ell_\lambda}> s_{\lambda,\beta}\big)=(1+o(1))\frac{\log1/\beta}{n^{1-\lambda}}.
$$
Since $m_\lambda=(1+o(1))n^{1-\lambda}$, we thus obtain, as $n\to\infty$, that
$$
\P\big(S^{\lambda}\le s_{\lambda,\beta}\big)=\Big(1-(1+o(1))\frac{\log1/\beta}{n^{1-\lambda}}\Big)^{m_\lambda}\to\exp\Big(-\log\frac1\beta\Big)=\beta,
$$
as required.
\end{proof}

\begin{lemma}\label{tribesquantile}
For every $\lambda,\beta\in(0,1)$ and $\eps>0$ small enough we have that any $\beta$-quantile $q_{\lambda,\beta}$ of $S^\lambda$ satisfies, for all sufficiently large $n$, that
 $$
 s_{\lambda,\beta-\eps}< q_{\lambda,\beta}<s_{\lambda,\beta+\eps}.
 $$
 \end{lemma}
  \begin{proof}
 Fix $\lambda,\beta \in (0,1)$ and $\eps>0$ such that $0<\beta-\eps<\beta+\eps<1$. By Lemma~\ref{case1conv}, for large $n$ we have
 $$
 \P\big(S^{\lambda}\le s_{\lambda,\beta-\eps}\big) \le \beta-\eps/2, $$
 which implies that $q_{\lambda,\beta}>s_{\lambda,\beta-\eps}$. We similarly obtain, again from Lemma~\ref{case1conv}, that
$$
\P\big(S^{\lambda} \ge  s_{\lambda,\beta+\eps}\big)\le\P\big(S^{\lambda} >  s_{\lambda,\beta+\eps/2}\big)\le 1-\beta-\eps/4,
$$
which shows that $q_{\lambda,\beta}<s_{\lambda,\beta+\eps}$, for large values of $n$.
\end{proof}
 
 
 With these estimates at hand, we now prove Proposition \ref{maintribesthm}.

\begin{proof}[Proof of Proposition~\ref{maintribesthm}]
Fix $\lambda,\beta\in(0,1)$ and let $q_{\lambda,\beta}$ be any $\beta$-quantile of $S^\lambda$. By Lemmas~\ref{case1conv} and~\ref{tribesquantile} we have for small enough $\eps>0$ and all large $n$ that
$$
\P\big(f_{\lambda,\beta}=0\big)=\P\big(S^\lambda\le q_{\beta,\lambda}\big)\le\P\big(S^\lambda\le s_{\lambda,\beta+\eps}\big)\le\beta+2\eps.
$$
Analogously we obtain the lower bound
$$
\P\big(f_{\lambda,\beta}=0\big)=\P\big(S^\lambda\le q_{\beta,\lambda}\big)\ge\P\big(S^\lambda\le s_{\lambda,\beta-\eps}\big)\ge\beta-2\eps.
$$
Since $\eps>0$ was arbitrary, it follows that $\P(f_{\lambda,\beta}=0)\to\beta$ as $n\to\infty$.

To prove that the sequence is noise sensitive we aim to prove that the sum of square influences tends to zero as $n\to\infty$. Noise sensitivity will then follow from the BKS Theorem.

First note that bits not part of any tribe have zero influence. In addition, all remaining influences are equal due to symmetry. It will hence suffice to bound the influence of the first bit of the first tribe. For this bit to be decisive there have to be precisely $\floor{q_{\lambda,\beta}}$ 1s among the remaining $\ell_\lambda-1$ bits of the first tribe, as well as no other tribe with more than $q_{\lambda,\beta}$ 1s. Since a particular tribe is unlikely to exceed $q_{\lambda,\beta}$, the probability of the latter approaches $\beta$ as $n\to\infty$. Consequently, by independence between tribes,
$$
\Inf_1(f_{\lambda,\beta})=(\beta+o(1))\P\big(X_{\ell_\lambda-1}=\floor{q_{\lambda,\beta}}\big),
$$
where $X_n$ again is a centred binomial of $n$ trials. Using~\eqref{eq:bin2} and Lemma~\ref{tribesquantile} we obtain for fixed values of $\lambda,\beta\in(0,1)$ that
$$
\Inf_1(f_{\lambda,\beta})\asymp\frac{\sqrt{\log n}}{n^{1-\lambda/2}}.
$$
Squaring the influences thus gives that
$$
\sum_{i\in[n]}\Inf_i(f_{\lambda,\beta})^2\asymp\frac{\log n}{n^{1-\lambda}}.
$$
For fixed $\lambda,\beta\in(0,1)$ the BKS Theorem hence implies that $f_{\lambda,\beta}$ is noise sensitive, as $n\to\infty$.
\end{proof}



\section{Cramér-type moderate deviations for triangular arrays}\label{sec:generalMD}

In this section we state and prove a Cram\'er-type result for the moderate deviations of a sum of independent random variables. The result is different from Cram\'er's classical result in that it applies to triangular arrays of independent, but not necessarily identically distributed, random variables. In particular, the distributions of the existing random variables are allowed to vary as more variables are included. As mentioned in the introduction, this will be one of the key steps to prove noise sensitivity in the context of first-passage percolation. 

\begin{theorem}\label{ldthm2}
 For every $m\ge1$, let $X^{(m)}_{1},X_2^{(m)},\ldots,X^{(m)}_{m}$ be a sequence of independent random variables with mean zero and finite variance, and set
 $$
 \sigma_m:= \sqrt{\frac{\sum_{i=1}^{m}\Var(X^{(m)}_{i})}{m}}.
 $$
 Suppose that $\sigma_m\ge1$ and that there exist global constants $\delta\in[0,1)$ and $C\ge1$ such that for every $m\ge1$ and $i=1,2,\ldots,m$, and all $j\ge2$, we have
 \begin{equation}\label{eq:sigma_cond}
 \E\big[\big|X_i^{(m)}\big|^{j}\big] \leq j!\,(C\sigma_m)^{(1+\delta)j}.
 \end{equation}
 Let $F_{m}$ be the distribution function of the normalised sum $(X^{(m)}_{1}+...+X^{(m)}_{m})/(\sigma_m\sqrt{m})$.  Then, assuming $\sigma_m^\delta\ll m^{1/6}$, we have for  $1\ll x\ll m^{1/6}/\sigma_m^{\delta}$ that
$$
1-F_{m}(x)= \Big[1+O\Big(\sigma_m^{3\delta}\frac{x^{3}}{\sqrt{m}}\Big)\Big]\big[1-\Phi(x)\big],
$$
where $\Phi$ denotes the distribution function of the standard Gaussian distribution.
 \end{theorem}

 Our proof will follow closely the proof of Cram\'er's Theorem as presented by Feller~\cite[Chapter~XVI.7]{feller-vol-2}. As is usual in the proof of theorems of this type, the proof will follow from the analysis of moment and cumulant generating functions. The moment generating function of a random variable $X$ is the function $f(s):=\E[e^{sX}]$, and the cumulant generating function is defined as $\psi(s):=\log f(s)$. These functions are not well defined for all random variables $X$, but when they are, in a vicinity of the origin, they provide useful information of the random variable.
 
 Before we tend to the proof of the above theorem, we prove a lemma regarding the regularity of the cumulant generating function of a random variable.

  \begin{lemma}\label{ldlma}
 Let $X$ be a random variable with mean zero, variance $\sigma^2$ and third moment $\mu_3$. Suppose there exists constants $\delta\ge0$ and $\gamma>0$ such that for all $j\ge2$
 \begin{equation}\label{eq:Xcond}
 \E\big[|X|^j\big]\le j!\,\gamma^{(1+\delta)j}.
 \end{equation}
 Then, the moment generating function $f(s)=\E[e^{sX}]$ and the cumulant generating function $\psi(s)=\log f(s)$ are well-defined and continuously differentiable of all orders for $|s|<1/\gamma^{1+\delta}$. Moreover, there exists a global constant $C>0$, not depending on the distribution of $X$, such that for $|s|\le1/(2\gamma^{1+\delta})$
 \begin{align*}
 \big|\psi(s)-\frac12\sigma^2s^2-\frac16\mu_3s^3\big|&\le C\gamma^{4(1+\delta)}|s|^4,\\
 \big|\psi'(s)-\sigma^2s-\frac12\mu_3s^2\big|&\le C\gamma^{4(1+\delta)}|s|^3,\\
 \big|\psi''(s)-\sigma^2-\mu_3s\big|&\le C\gamma^{4(1+\delta)}|s|^2.
 \end{align*}
 \end{lemma}

 \begin{proof}
Let $\mu_j:=\E[X^j]$ denote the $j$th moment of $X$, so that, in particular, $\mu_1=0$, $\mu_2=\sigma^2$ and $\mu_3$ is the third moment, as defined above. Then, by~\eqref{eq:Xcond}, we have for $k\ge1$ and $|s|\le1/(2\gamma^{1+\delta})$
 \begin{equation}\label{eq:f_expansion}
 \Big|f(s)-1-\sum_{j=2}^k\frac{\mu_j}{j!}s^j\Big|\le\sum_{j\ge k+1}\frac{|\mu_j|}{j!}|s|^j\le\sum_{j\ge k+1}\big(\gamma^{1+\delta}|s|\big)^j\le2\big(\gamma^{1+\delta}|s|\big)^{k+1}.
 \end{equation}
In particular, it follows from the second-to-last expression in~\eqref{eq:f_expansion} that $f$ is analytic for $|s|<1/\gamma^{1+\delta}$, and hence well-defined and continuously differentiable of all orders. In the same way we find that there exist global constants $C',C''>0$ such that for $|s|\le1/(2\gamma^{1+\delta})$
 \begin{equation}\label{eq:fp_expansion}
     \big|f'(s)-\sigma^2s-\frac12\mu_3s^2\big|\le C'\gamma^{4(1+\delta)}|s|^3\quad\text{and}\quad\big|f''(s)-\sigma^2-\mu_3s\big|\le C''\gamma^{4(1+\delta)}|s|^2.
 \end{equation}

By~\eqref{eq:f_expansion} we have, in particular, that $|f(s)-1|\le2(\gamma^{1+\delta}|s|)^2$ and $|f(s)-1-\frac12\sigma^2s^2-\frac16\mu_3s^3|\le2(\gamma^{1+\delta}|s|)^4$. Since $\psi(s)=\log f(s)$ and $\log(1+x)=x+O(x^2)$ we obtain for $|s|\le1/(2\gamma^{1+\delta})$ that
$$
\big|\psi(s)-\frac12\sigma^2s^2-\frac16\mu_3s^3\big|\le\big|\psi(s)-(f(s)-1)\big|+\big|f(s)-1-\frac12\sigma^2s^2-\frac16\mu_3s^3\big|\le C(\gamma^{1+\delta}|s|)^4,
$$
for some constant $C$ not depending on the distribution of $X$. Moreover, differentiation yields
 $$
 \psi'(s)=\frac{f'(s)}{f(s)}\quad\text{and}\quad\psi''(s)=\frac{f''(s)f(s)-f'(s)^2}{f(s)^2}.
 $$
 Hence, since $\frac{1}{1+x}=1+O(x)$, and since $|f'(s)|\le4\gamma^{2(1+\delta)}|s|$, we obtain by~\eqref{eq:fp_expansion} that for $|s|\le1/(2\gamma^{1+\delta})$
 $$
 \big|\psi'(s)-\sigma^2s-\frac12\mu_3s^2\big|\le |f'(s)|\Big|\frac{1}{f(s)}-1\Big|+\big|f'(s)-\sigma^2s-\frac12\mu_3s^2\big|\le C\gamma^{4(1+\delta)}|s|^3,
 $$
 for some global constant $C$ not depending on the distribution of $X$. Finally, using~\eqref{eq:fp_expansion}, and that $|f''(s)|\le4\gamma^{2(1+\delta)}$ and $|\psi'(s)|\le4\gamma^{2(1+\delta)}|s|$, we obtain that for $|s|\le1/(2\gamma^{1+\delta})$
 $$
 \big|\psi''(s)-\sigma^2-\mu_3s\big|\le |f''(s)|\Big|\frac{1}{f(s)}-1\Big|+\big|f''(s)-\sigma^2-\mu_3s\big|+|\psi'(s)|^2\le C\gamma^{4(1+\delta)}|s|^2,
 $$
 for some global constant $C$ not depending on the distribution of $X$.
 \end{proof}

 \begin{proof}[Proof of Theorem~\ref{ldthm2}]
Although we will be working with triangular arrays, where the distribution of all variables are allowed to change in each step, we shall throughout the proof suppress the dependence on $m$ in order to ease the notation. For instance, we shall for a given value of $m\ge1$ and $i=1,2,\ldots,m$ denote by $G_i$ the distribution function of $X_i=X_i^{(m)}$, although the distribution is allowed to vary with $m$. Moreover, we shall let $f_i(s):=\E[e^{sX_i}]$ denote the moment generating function and $\psi_i(s):=\log f_i(s)$ the cumulant generating function of $G_i$. From Lemma~\ref{ldlma}, by assumption~\eqref{eq:sigma_cond}, it follows that $f_i(s)$ and $\psi_i(s)$ are well-defined and smooth for $|s|<(C\sigma_m)^{-(1+\delta)}$, and we let
 $$
 \psi(s):=\frac{1}{m}\sum_{i=1}^{m}\psi_{i}(s).
 $$

Note that for each $m\ge1$ and $1\le i\le m$ the first two derivatives of $\psi_{i}$ are again given by
$$
\psi_{i}'(s)=\frac{f_{i}'(s)}{f_{i}(s)}\quad\text{and}\quad\psi_{i}''(s)=\frac{f_{i}''(s)f_{i}(s)-f'_{i}(s)^2}{f(s)^2}.
$$
An application of the Cauchy-Schwarz inequality shows that
$$
f'_i(s)^2=\E\big[X_ie^{sX_i}\big]^2\le\E\big[|X_i|e^{sX_i}\big]^2\le\E\big[X_i^2e^{sX_i}\big]\E\big[e^{sX_i}\big]=f''_i(s)f_i(s),$$
so that $\psi_{i}''(s)\geq0$ on the domain where it is defined. In fact, since $G_i$ has mean zero, the first inequality is strict and $\psi_{i}''(s)>0$, unless $G_i$ also has zero variance. Since $\sigma_m^2\ge1$ by assumption, it follows that not all $G_i$ may have zero variance, and so that
$$
\psi''(s)=\frac{1}{m}\sum_{i=1}^{m}\psi''_{i}(s)>0
$$
on its domain. Since $\psi_{i}'(0)=0$ for each $i$ it follows that $\psi'(s)$ is positive and strictly increasing on the interval $(0,1/(C\sigma_m)^{1+\delta})$. Consequently, for $s>0$ and $x>0$ the relation
 \begin{equation}\label{eq:s-x_link}
 \sqrt{m}\psi'(s)=\sigma_m x, 
 \end{equation}
establishes a 1-1 correspondence between $s$ and $x$.
From Lemma~\ref{ldlma} we obtain that
$$
\Big|\frac{x}{\sigma_m\sqrt{m}}-s\Big|=\Big|\frac{1}{\sigma_m^2}\psi'(s)-s\Big|=O(\sigma_m^{1+3\delta}|s|^2),
$$
so that for $s=o(1/\sigma_m^{1+3\delta})$ we have
\begin{equation}\label{eq:s-x_approx}
\frac{x}{\sigma_m\sqrt{m}}=(1+o(1))s.
\end{equation}
We shall henceforth assume that $x$ and $s$ satisfy~\eqref{eq:s-x_link} and that $s=o(1/\sigma_m^{1+3\delta})$, so that also~\eqref{eq:s-x_approx} holds.

Following the steps of Feller, we next associate a new probability distribution $V_i$ with the distribution $G_i$ defined by 
\begin{equation}\label{eq:associated2}
V_i(dy) = e^{-\psi_{i}(s)}e^{sy}\,G_i(dy),
\end{equation}
where $s$ is chosen accordingly to our previous restrictions. Analogously to the function $f_{i}$, we define the moment generating function of $V_i$ as 
$$
\nu_{i}(\zeta):=\int e^{\zeta y}\,V_i(dy) = \frac{f_{i}(\zeta+s)}{f_{i}(s)}.
$$
It follows by differentiation that $V_i$ has expectation $\psi_{i}'(s)$ and variance $\psi_{i}''(s)$. Now, let $F_{m}^{\star}$ denote the the non-normalized version of $F_{m}$, i.e.\ the cumulative distribution function of the sum of the $m$ independent variables distributed as $G_1,\ldots,G_m$, and let $U_m^{\star}$ denote ditto for $m$ independent variables distributed as $V_1,\ldots,V_m$. Then $U_m^{\star}$ has expectation $m\psi'(s)$ and variance $m\psi''(s)$. Also, by comparing the moment generating functions, we observe that $U_m^\star$ and $F_{m}^{\star}$ satisfy a relation similar to~\eqref{eq:associated2} in that
$$
U_m^\star(dy) = e^{-m\psi(s)}e^{sy}\,F_m^\star(dy).
$$
It follows that
\begin{equation}\label{errorequation2}
    1-F_{m}(x) = 1-F_{m}^{\star}(x\sigma_m\sqrt{m}) = e^{m\psi(s)}\int_{x\sigma_m\sqrt{m}}^{\infty}e^{-sy}\,U_m^{\star}(dy).
\end{equation}

The proof will now proceed in two steps. We first analyse the expression obtained from~\eqref{errorequation2} when substituting $U_m^\star$ by the normal distribution with the same mean and variance. Second, we evaluate the relative error committed by this operation.
So, we define $A_{s}$ to be the quantity obtained by substituting $U_m^{\star}$ by $N(m\psi'(s),m\psi''(s))$ in the right-hand side of~\eqref{errorequation2}. Using the substitution of variables $y=m\psi'(s)+z\sqrt{m\psi''(s)}$, and the relation in~\eqref{eq:s-x_link}, we have that
\begin{equation}\label{eq:As3}
\begin{split}
    A_{s} &:= e^{m\psi(s)}\int_{x\sigma_m\sqrt{m}}^{\infty}e^{-sy}\frac{1}{\sqrt{2\pi m \psi''(s)}}e^{- (y-m\psi'(s))^{2}/(2m\psi''(s))}\,dy\\ 
    &\;=e^{m[\psi(s)-s\psi'(s)+\frac{1}{2}s^{2}\psi''(s)]}\frac{1}{\sqrt{2\pi}}\int_{0}^{\infty}e^{-(z+s\sqrt{m\psi''(s)})^{2}/2}\,dz.
    \end{split}
\end{equation}
We are now interested in the behavior of
$$
m\Big[\psi(s)-s\psi'(s)+\frac{1}{2}s^{2}\psi''(s)\Big]=\sum_{i=1}^m\Big[\psi_i(s)-s\psi_i'(s)+\frac{1}{2}s^{2}\psi_i''(s)\Big].
$$
Lemma~\ref{ldlma}, applied to each term in the sum, gives that for $s=o(1/\sigma_m^{1+3\delta})$
\begin{equation}\label{eq:exp_again}
m\Big[\psi(s)-s\psi'(s)+\frac{1}{2}s^{2}\psi''(s)\Big]=\frac{m}{6}\mu_3s^3+O(m\sigma_m^{4(1+\delta)}s^4)=O(m\sigma_m^{3(1+\delta)}s^3),
\end{equation}
where $\mu_3$ is the average of the third moments of the distributions $G_1,\ldots,G_m$. The above expression vanishes as $m\to\infty$ if $s=o(1/(m^{1/3}\sigma_m^{1+\delta}))$. We note that, under the assumption that $\sigma_m^\delta=o(m^{1/6})$, which is assumed, $s=o(1/(m^{1/3}\sigma_m^{1+\delta}))$ is a stronger condition than $s=o(1/\sigma_m^{1+3\delta})$.

Since $e^y=1+O(y)$ for small values of $|y|$, it follows from~\eqref{eq:As3} and~\eqref{eq:exp_again} that for $s=o(1/(m^{1/3}\sigma_m^{1+\delta}))$
$$
A_s=\Big(1+O\big(m\sigma_m^{3(1+\delta)}s^3\big)\Big)\big[1-\Phi(\bar x)\big],
$$
where $\bar x:=s\sqrt{m\psi''(s)}$. Hence, if $s=o(1/(m^{1/3}\sigma_m^{1+\delta}))$, which by~\eqref{eq:s-x_approx} is equivalent to $x=o(m^{1/6}/\sigma_m^{\delta})$, we obtain from~\eqref{eq:s-x_approx} that
\begin{equation}\label{eq:As4}
A_{s} = \bigg[1+O\bigg(\sigma_m^{3\delta}\frac{x^{3}}{\sqrt{m}}\bigg)\bigg]\big[1-\Phi(\bar{x})\big].
\end{equation}

We now want to verify that we can substitute $\bar{x}$ by $x$ in~\eqref{eq:As4}. Observe that, by~\eqref{eq:s-x_link}, we have
$$
|x-\bar{x}|= \sqrt{m}\,\Big|\frac{1}{\sigma_m}\psi'(s)-s\sqrt{\psi''(s)}\Big|,
$$
Using that $\sqrt{1+y}=1+\frac12y+O(y^2)$ for $|y|$ small, we obtain from Lemma~\ref{ldlma} that for $s=o(1/\sigma_m^{1+3\delta})$
$$
s\sqrt{\psi''(s)}=\sigma_ms\sqrt{1+\frac{\mu_3}{\sigma_m^2}s+O(\sigma_m^{2+4\delta}s^2)}=\sigma_ms+\frac{\mu_3}{2\sigma_m}s^2+O(\sigma_m^{3+4\delta}s^3)+O(\sigma_m^{3+6\delta}s^3).
$$
Another application of Lemma~\ref{ldlma} thus gives, for $s=o(1/\sigma_m^{1+3\delta})$,
$$
\Big|\frac{1}{\sigma_m}\psi'(s)-s\sqrt{\psi''(s)}\Big|=O(\sigma_m^{3+6\delta}s^3).
$$
Hence, for $s=o(1/(m^{1/3}\sigma_m^{1+\delta}))$, which by~\eqref{eq:s-x_approx} is equivalent to $x=o(m^{1/6}/\sigma_m^{\delta})$,
\eqref{eq:s-x_link} gives
\begin{equation}\label{eq:x-xbar}
|x-\bar x|=\sqrt{m}\,O(\sigma_m^{3+6\delta}s^3)=O\bigg(\sigma_m^{6\delta}\frac{x^{3}}{m}\bigg),
\end{equation}
from which we conclude that $|x-\bar x|=o(x)$.


Denote by $\varphi(y)$ the density of the standard normal distribution. Recall that as $y\rightarrow \infty$
\begin{equation}\label{eq:normal2}
\frac{\varphi(y)}{1-\Phi(y)}  = (1+o(1))y.
\end{equation}
Integrating the above expression between $x$ and $\bar{x}$ we find via~\eqref{eq:x-xbar} that for $x \rightarrow \infty$ such that 
$x=o(m^{1/6}/\sigma_m^{\delta})$,
$$ 
\Big|\log \frac{1-\Phi(\bar{x})}{1-\Phi(x)}\Big| = O\big(|x+\bar{x}| | \bar{x}-x|\big) = O\Big(\sigma_m^{6\delta}\frac{x^{4}}{m}\Big)=O\Big(\sigma_m^{3\delta}\frac{x^{3}}{\sqrt{m}}\Big),
$$
and finally, since $e^y=1+O(y)$ for $|y|$ small, we obtain that 
$$
\frac{1-\Phi(\bar{x})}{1-\Phi(x)} =
1+O\Big(\sigma_m^{3\delta}\frac{x^{3}}{\sqrt{m}}\Big).
$$
Now, together with $\eqref{eq:As4},$ we have that if $x\rightarrow\infty$ with $x=o(m^{1/6}/\sigma_m^{\delta})$, then
\begin{equation}\label{concludingeq2}
A_{s} = \bigg[1+O\bigg(\sigma_m^{3\delta}\frac{x^{3}}{\sqrt{m}}\bigg)\bigg]\big[1-\Phi(x)\big].
\end{equation}

It remains to estimate the error committed by substituting $U_m^\star$ by the $N(m\psi'(s),m\psi''(s))$ distribution in the right-hand side of~\eqref{errorequation2}. Let $Y_i$ denote a generic random variable distributed according to $V_i$. Recall that $Y_i$ has mean $\psi_{i}'(s)$ and variance $\psi_{i}''(s)$.
Let $\Phi_{s}$ denote the cumulative distribution function of the $N(m\psi'(s),m\psi''(s))$ distribution. By the Berry-Esseen Theorem (for non-identically distributed variables) we have that for all $y$ that
\begin{equation}\label{eq:BE}
\big|U_{m}^{\star}(y)-\Phi_{s}(y)\big|\le 3\big(m\psi''(s)\big)^{-3/2}\sum_{i=1}^{m}\E\big[|Y_i-\psi'_{i}(s)|^{3}\big]. 
\end{equation}
Integration by parts, using~\eqref{eq:s-x_link},~\eqref{errorequation2} and~\eqref{eq:BE}, gives
\begin{equation*}
\begin{split}
\big|1-F_{m}(x)-A_{s}\big|&= e^{m\psi(s)}\bigg|\int_{x\sigma_m\sqrt{m}}^\infty e^{-sy}\,U_m^{\star}(dy)-\int_{x\sigma_m\sqrt{m}}^\infty e^{-sy}\,\Phi_{s}(dy)\bigg| \\
&\le e^{m\psi(s)}\bigg(-\Big[e^{-sy}\big|U^{*}_{m}(y)-\Phi_{s}(y)\big|\Big]_{m\psi'(s)}^{\infty} + s\int_{m\psi'(s)}^{\infty}e^{-sy}\big|U^{*}_{m}(y)-\Phi_{s}(y)\big|\,dy\bigg) \\
 &\le6\big(m\psi''(s)\big)^{-3/2}\, e^{m[\psi(s)-s\psi'(s)]}\,\sum_{i=1}^{m}\E\big[|Y_{i}-\psi'_{i}(s)|^{3}\big].
\end{split}
\end{equation*}
We may bound the central absolute third moment, for $s=o(1/\sigma_m^{1+3\delta})$, as
$$
\E\big[|Y_i-\psi_{i}'(s)|^{3}\big]\le 2^3\big(\E\big[|Y_i|^3\big]+|\psi_{i}'(s)|^3\big)\le 8\,\E\big[Y_i^4\big]^{3/4}+O\big(\sigma_m^{3(1-\delta)}\big),
$$
where the second inequality follows from Jensen's inequality. By an expansion similar as before, we obtain for $s =o\big(1/\sigma_m^{1+3\delta}\big)$ that 
$f_i(s)=1+o(1)$ and $f_i^{(4)}(s)=O(\sigma_m^{4(1+\delta)})$, and hence that
$$
\E\big[Y_i^{4}\big]=\frac{f_{i}^{(4)}(s)}{f_{i}(s)}=O(\sigma_m^{4(1+\delta)}).
$$
Moreover, for $s =o\big(1/\sigma_m^{1+3\delta}\big)$, we have $\psi''(s)=\sigma_m^{2}(1+O(\sigma_m^{1+3\delta}s))=\sigma_m^2(1+o(1))$, which gives
\begin{equation}\label{eq:hendricks}
 \big|1-F_{m}(x)-A_{s}\big| = O(\sigma_m^{3\delta}/\sqrt{m})\cdot e^{m[\psi(s)-s\psi'(s)]}.
\end{equation}
 
Next, we recall from~\eqref{eq:As3} and the definition of $\bar x$ that
$$
A_{s} = e^{m[\psi(s)-s\psi'(s)]}\,e^{\bar{x}^{2}/2}\big[1-\Phi(\bar{x})\big]. 
$$
By~\eqref{eq:normal2}, and the observation that $x/\bar x=1+o(1)$ for $x=o(m^{1/6}/\sigma_m^{\delta})$, we find that
$$
A_{s}=(1+o(1))\frac{1}{\sqrt{2\pi}x}e^{m[\psi(s)-s\psi'(s)]}, 
$$
and hence, together with~\eqref{eq:hendricks}, that
$$
\big|1-F_{m}(x)-A_{s}\big|=O\Big(\sigma_m^{3\delta}\frac{x}{\sqrt{m}}\Big)A_{s}.
$$
In conclusion,
$$
1-F_{m}(x) = \Big[1+O\Big(\sigma_m^{3\delta}\frac{x}{\sqrt{m}}\Big)\Big]A_{s}, 
$$
which, together with~\eqref{concludingeq2} completes the proof.
\end{proof}

\section{Moderate deviations in first-passage percolation}\label{SecDistribution}

We now proceed to derive a moderate deviations theorem for first-passage percolation across thin rectangles. The result will follow from the moderate deviations theorem for triangular arrays (Theorem~\ref{ldthm2}) via an approach of Chatterjee and Dey~\cite{Chatterjee2009CentralLT}.

Recall the definition, in~\eqref{eq:t_rect}, that $t(n,k)$ denotes the left-right crossing time of the rectangle $R(n,k)=[0,n]\times[0,k]$. By the rectangle being `thin' refers to the height satisfying $k\ll n^\alpha$ for some $\alpha<2/3$. For the proof to go through, we will have to restrict the height even further.

Since we shall foremost be interested in the lower tail, i.e.\ deviations of $t(n,k)$ below its mean, we state the theorem accordingly. An analogous statement holds for the upper tail, i.e.\ for deviations above the mean.

\begin{theorem}\label{thm:mdfpp}
Suppose that $F$ is uniformly distributed on $\{a,b\}$ for some $0<a<b<\infty$.
Let $\alpha<1/10$ and suppose that $k_n\ll n^\alpha$. Then, for $1\ll x\ll n^{(1-10\alpha)/18}$ we have
\begin{equation*}
\P\Big(t(n,k_n)- \E[t(n,k_n)] <-x\sqrt{\Var(t(n,k_n))}\Big)= \Big[1-\Phi(x)\Big]\bigg[1+o\bigg(\frac{x^{3}}{n^{(1-10\alpha)/6}}\bigg)\bigg].
\end{equation*}
Moreover, the analogous statement holds for the upper tail.
\end{theorem}

While we here focus on weight distributions supported on two points, we mention that the proof of the above theorem goes through without modification for bounded weight distributions for which the lower bound in Lemma~\ref{lma:CD_variance} remains valid. (In general, the necessary condition for this is that $F(0)<1/2$. When $(k_n)_{n\ge1}$ is bounded, no additional condition is needed.)


\subsection{First-passage percolation across thin rectangles}

First-passage percolation on rectangular subsets of the square lattice have previously been considered by Ahlberg~\cite{ahl15} and Chatterjee and Dey~\cite{Chatterjee2009CentralLT}. In both papers the authors prove asymptotic normality for the crossing time of thin rectangles, though by different means. In~\cite{ahl15} the author adopts a regenerative approach that applies for fixed $k$, but fails for rectangles with height growing polynomially in $n$. In~\cite{Chatterjee2009CentralLT} the authors develop a different approximation scheme that works for rectangles with height $k_n=o(n^\alpha)$ for some $\alpha<1/3$. It is the latter approach, from~\cite{Chatterjee2009CentralLT}, that will be of interest to us here, as it will apply to rectangles of growing height.

The idea from~\cite{Chatterjee2009CentralLT} is to chop the rectangle $[0,n]\times[0,k-1]$ up into smaller pieces, and approximate the crossing time of the original rectangle with the sum of the crossing times of the shorter stubs. This approximates the crossing time of the original rectangle with a sum of independent variables with roughly the same distribution. Chatterjee and Dey show in~\cite{Chatterjee2009CentralLT} that if the number of independent variables is large in comparison to the width of the original rectangle, then the error committed in the approximation can be controlled.

We shall below adopt their approach in the proof of Theorem~\ref{thm:mdfpp}. Their argument will here require a somewhat stronger restriction on the rate at which the rectangle grows. This restriction arises from the gap in upper and lower bounds on the moments of the crossing times, which will force us to apply Theorem~\ref{ldthm2} with some $\delta>0$. The following two lemmas (from~\cite{Chatterjee2009CentralLT}) bound the central moments on rectangle crossing times, and will be used in the proof of Theorem~\ref{thm:mdfpp}.

\begin{lemma}\label{lma:CD_variance}
Then there exists $c>0$ such that for all $n\ge k\ge1$
\begin{equation*}
    c\frac{n}{k}\le\Var\big(t(n,k)\big)\le\frac1cn.
\end{equation*}
\end{lemma}

\begin{proof}
    The upper bound follows from Proposition~5.1 in~\cite{Chatterjee2009CentralLT}. The lower bound is a slight sharpening of Lemma~4.1 in~\cite{Chatterjee2009CentralLT}. Indeed, the proof of that lemma gives that
    $$
    \Var\big(t(n,k)\big)\ge\frac{1}{2kn}\E\bigg[\sum_{e\in\pi(n,k)}\E[(\omega_e-\eta)_+|\omega_e]\bigg]^2,
    $$
    where $\pi(n,k)$ denotes the set of edges belonging to some geodesic for $t(n,k)$, $\eta$ is an independent copy of $\omega_e$, and $x_+=\max\{x,0\}$. When $F$ put equal mass to $a$ and $b$, then
    $$
    \E[(\omega_e-\eta)_+|\omega_e]=\frac{\omega_e-a}{2},
    $$
    leading to the bound
    $$
    \Var\big(t(n,k)\big)\ge\frac{(b-a)^2}{8kn}\E\big[\#\{e\in\pi(n,k):\omega_e=b\}\big]^2.
    $$
    Hence, to complete the proof, it will suffice to argue that there exists a constant $c>0$ such that for all $n\ge k\ge1$ 
    \begin{equation}\label{eq:empiric}
    \E\big[\#\{e\in\pi(n,k):\omega_e=b\}\big]\ge cn.
    \end{equation}

    In~\cite{Chatterjee2009CentralLT}, the authors relied on standard percolation results to argue that when $F(a)<1/2$, then $\pi(n,k)$ would have to pick up a large number of weight $b$ edges, and hence that~\eqref{eq:empiric} holds. In Proposition~\ref{prop:empirical}, stated and proved in the appendix of this paper, we show that~\eqref{eq:empiric} remains true also when $F(a)=1/2$, which is the case under consideration here, by revisiting an argument of van de Berg and Kesten~\cite{vdBK93}. This completes the proof of the lemma.
\end{proof}

The next result, also from~\cite{Chatterjee2009CentralLT}, is a bound on central moments.

\begin{lemma}\label{lma:CD_moments}
There exists $C>0$ such that for all $j\ge2$, $n\ge1$ and $k\le\sqrt{n}$ we have
\begin{equation*}
    \E\Big[\big|t(n,k)-\E[t(n,k)]\big|^j\Big]\le(Cj)^jn^{j/2}.
\end{equation*}
\end{lemma}

\begin{proof}
    This is more precise version of Proposition~5.1 in~\cite{Chatterjee2009CentralLT}, which is obtained by combining Lemmas~5.4 and~5.5 of the same paper.
\end{proof}

\subsection{Proof of Theorem~\ref{thm:mdfpp}}

Fix $\alpha<1/10$ and suppose that $k_n\le n^\alpha$ for large values of $n$. Let $\gamma\in(1/2,1)$ be a parameter to be determined later, and set $m_n:=\lfloor n^{1-\gamma}\rfloor$ and $\ell_n:=\lfloor n/m_n\rfloor$. We partition the interval $[0,n]$ into $m_n$ subintervals of length either $\ell_n$ or $\ell_n+1$ (where consecutive intervals share endpoints). Denote the intervals by $I_1,I_2,\ldots,I_{m_n}$ and let $Y_i$ denote the left-right crossing time of the rectangle $I_i\times[0,k_n-1]$. Since the intervals are disjoint (except for their boundary points) the resulting variables $Y_1,Y_2,\ldots,Y_{m_n}$ are independent and distributed as $t(\ell_n,k_n)$ or $t(\ell_n+1,k_n)$, depending on the length of the corresponding interval. 

Since every path crossing $[0,n]\times[0,k_n-1]$ from left to right can be partitioned into paths crossing the intervals $I_1,I_2,\ldots,I_{m_n}$, it follows that the sum of the $Y_i$' s is a lower bound on $t(n,k_n)$.
Moreover, since the edge weights are bounded by $b>0$, and since there are no more that $k_n$ edges along the boundary between two consecutive rectangles $I_i\times[0,k_n-1]$ and $I_{i+1}\times[0,k_n-1]$, we obtain that
\begin{equation}\label{eq:Tkn_approx}
\sum_{i=1}^{m_n}Y_i\;\le\; t(n,k_n)\;\le\;\sum_{i=1}^{m_n}Y_i+
b\,m_nk_n.
\end{equation}
Let $X_i:=Y_i-\E[Y_i]$ be the centered version of $Y_i$, and set $S_n:=\sum_{i=1}^{m_n}X_i$. Taking expectation in~\eqref{eq:Tkn_approx}, and subtracting the result from the same, yields
\begin{equation}\label{eq:Tkn-E_approx}
    S_n-b\,m_nk_n\;\le\; t(n,k_n)-\E[t(n,k_n)]\;\le\;S_n+b\,m_nk_n.
\end{equation}

For $n\ge1$, let
$$
\sigma_n:=\sqrt{\frac{1}{m_n}\Var(S_n)}.
$$
Since $\ell_n\sim n^\gamma$ and $k=o(n^\alpha)$, it follows from Lemma~\ref{lma:CD_variance} that there exists $c>0$ so that for all $n\ge1$
$$
c\,n^{\gamma-\alpha}\le\Var(X_i)\le\frac1c\,n^\gamma,
$$
and hence that
\begin{equation}\label{eq:sigma_bound}
\sqrt{c}\,n^{(\gamma-\alpha)/2}\le\sigma_n\le\frac{1}{\sqrt{c}}\,n^{\gamma/2}.
\end{equation}
Moreover, by the reverse triangle inequality,
$$
\bigg|\sqrt{\E\big[(t(n,k_n)-\E[t(n,k_n)])^{2}\big]}- \sqrt{\E\big[S_n^{2}\big]} \bigg| \leq \sqrt{\E\Big[\big(t(n,k_n)-\E[t(n,k_n)] -S_n\big)^2\Big]}\le b\,m_nk_n,
$$
which with~\eqref{eq:sigma_bound} gives
\begin{equation}\label{eq:var_sigm}
\bigg|\frac{\sqrt{\Var(t(n,k_n))}}{\sigma_n\sqrt{m_n}}-1\bigg|\le\frac{b\,k_n\sqrt{m_n}}{\sqrt{c}\,n^{(\gamma-\alpha)/2}}=O\Big(\frac{1}{n^{(2\gamma-1-3\alpha)/2}}\Big).
\end{equation}
The above is $o(1)$ under the condition that $2\gamma>1+3\alpha$.

From Lemma~\ref{lma:CD_moments} and~\eqref{eq:sigma_bound} we obtain in turn (since $\gamma>2\alpha$) that
\begin{equation}\label{eq:mom_bound}
    \E\big[|X_i|^j\big]\le(Cj)^jn^{\gamma j/2}\le j!(Ce)^j\sigma_n^{(1+\alpha/(\gamma-\alpha))j}.
\end{equation}
In particular, this means that $X_1,X_2,\ldots,X_{m_n}$ satisfy~\eqref{eq:sigma_cond} with $\delta=\alpha/(\gamma-\alpha)$. We note, in addition, that
\begin{equation}\label{eq:error_b}
\frac{\sigma_n^{\alpha/(\gamma-\alpha)}}{m_n^{1/6}}=O\Big(\frac{n^{(\gamma/2)\alpha/(\gamma-\alpha)}}{n^{(1-\gamma)/6}}\Big)
=O\Big(\frac{1}{n^{(1-\gamma)/6-\alpha\gamma/[2(\gamma-\alpha)]}}\Big).
\end{equation}

Let $\beta_1:=(2\gamma-1-3\alpha)/2$ and $\beta_2:=(1-\gamma)/6-\alpha\gamma/[2(\gamma-\alpha)]$ denote the exponents in the right-hand sides of~\eqref{eq:var_sigm} and~\eqref{eq:error_b}, respectively. Now, set $\gamma=2/3$. This gives
$$
\beta_1=\frac{1-9\alpha}{6}\quad\text{and}\quad\beta_2>\frac{1-10\alpha}{18},
$$
which for $\alpha<1/10$ are strictly positive. Hence, Theorem~\ref{ldthm2} applies and gives that for $1\ll x\ll n^{\beta_2}$ that
\begin{equation}\label{eq:md_app}
    \P\bigg(\frac{S_n}{\sigma_n\sqrt{m_n}}<-x\bigg)=\Big[1+O\Big(\frac{x^3}{n^{3\beta_2}}\Big)\Big]\big[1-\Phi(x)\big].
\end{equation}

Now, let
$$
x_\pm:=x\bigg(1\pm\frac{b\,m_nk_n}{x\sqrt{\Var(t(n,k_n))}}\bigg)\frac{\sqrt{\Var(t(n,k_n))}}{\sigma_n\sqrt{m_n}}.
$$
By~\eqref{eq:Tkn-E_approx} we see that
$$
\P\bigg(\frac{t(n,k_n)-\E[t(n,k_n)]}{\sqrt{\Var(t(n,k_n))}}<-x\bigg)\le\P\bigg(\frac{S_n}{\sigma_n\sqrt{m_n}}<-x_-\bigg)=\Big[1+O\Big(\frac{x_-^3}{n^{3\beta_2}}\Big)\Big]\big[1-\Phi(x_-)\big].
$$
Analogously, we obtain
$$
\P\bigg(\frac{t(n,k_n)-\E[t(n,k_n)]}{\sqrt{\Var(t(n,k_n))}}<-x\bigg)\ge\P\bigg(\frac{S_n}{\sigma_n\sqrt{m_n}}<-x_+\bigg)=\Big[1+O\Big(\frac{x_+^3}{n^{3\beta_2}}\Big)\Big]\big[1-\Phi(x_+)\big].
$$

Finally, by~\eqref{eq:var_sigm}, we have
$$
x_\pm=x\Big(1+O\Big(\frac{1}{n^{\beta_1}}\Big)\Big),
$$
which give $1-\Phi(x_\pm)=[1+O(n^{-\beta_1})][1-\Phi(x)]$, which by the established bounds on $\beta_1$ and $\beta_2$ completes the proof.

\section{The analysis of independent rectangle crossings}\label{SecSideThm}

In this section we formulate and prove a preliminary version of our main theorem, concerning the minimum crossing time of a large number of {\em disjoint} rectangles. Since the rectangles are disjoint, the corresponding crossing times are independent, which facilitates the analysis.

Recall that $t(n,k)$ denotes the crossing time of the rectangle $[0,n]\times[0,k-1]$, and that $t_i(n,k)$ denotes the translation of $t(n,k)$ along the vector $(0,i)$, so that $t_i(n,k)$ is the crossing time of the rectangle $[0,n]\times[i,i+k-1]$. Finally, set
$$
\tau^\star(n,k):=\min\big\{t_{(i-1)k}(n,k):i=1,2,\ldots,\lfloor n/k\rfloor\big\}.
$$
Note that the different rectangles are disjoint and together tile the square $[0,n]\times[0,n-1]$. Consequently, $\tau^\star(n,k)$ is the minimum of $\lfloor n/k\rfloor$ independent copies of $t(n,k)$, and this independence will facilitate the analysis of $\tau^\star(n,k)$. We remark that for the choice $k=1$ we have $\tau^\star(n,k)=\tau(n,k)$, and the two are equivalent to the polynomial Tribes function on $n^2$ bits with $\lambda=1/2$.

\begin{theorem}\label{thm:preliminary}
    Suppose that $F$ is supported on $\{a,b\}$ for some $0<a<b<\infty$.
    Suppose that $k_n=o(n^{1/22})$. For every $\beta\in(0,1)$, and any $\beta$-quantile $q_\beta$ of $\tau^\star(n,k_n)$, we have
    $$
    \lim_{n\to\infty}\P\big(\tau^\star(n,k_n)<q_\beta\big)=1-\lim_{n\to\infty}\P\big(\tau^\star(n,k_n)>q_\beta\big)=\beta,
    $$
    and the function $f_n^\star:=\ind_{\{\tau^\star(n,k_n)>q_\beta\}}$ is noise sensitive.
\end{theorem}

We remark that the above theorem remains true for $k_n=o(n^{1/11})$. However, in order to be able to use one of the lemmas below (Lemma~\ref{lma:interval}) also in the proof of Theorem~\ref{thm:main} in the next section, we impose the restriction $k_n=o(n^{1/22})$ for the conclusion of the lemma to be stronger.

Similarly as in Section~\ref{sectribes}, when analysing the generalised tribes function, we split the proof of Theorem~\ref{thm:preliminary} into several lemmas. These lemmas will also be important in the deduction of Theorem~\ref{thm:main} in the next section. 

Given a sequence $(k_n)_{n\ge1}$, set $\ell_n:=\lfloor n/k_n\rfloor$. For $\beta\in(0,1)$ let
$$
d(n,\beta):=\sqrt{2\log\ell_n-\log(2\log\ell_n)-2\log\big(\sqrt{2\pi}\log(1/\beta)\big)},
$$
and set
\begin{equation}\label{eq:t_beta}
u_\beta=u_\beta(n):=\E[t(n,k_n)]-d(n,1-\beta)\cdot\sqrt{\Var(t(n,k_n))}.
\end{equation}
The first couple of lemmas determine the asymptotic growth of the quantiles of $\tau^\star(n,k_n)$.

 \begin{lemma}\label{lma:nottrivial}
Suppose that $k_n=o(n^{1/22})$. For every $\beta\in(0,1)$, with $u_\beta$ as defined in~\eqref{eq:t_beta}, we have $$\lim_{n\to\infty}\P\big(\tau^\star(n,k_n)<u_\beta\big)=\beta.$$
 \end{lemma}
 
  \begin{proof}
Note first that due to independence we have
$$
\P\big(\tau^\star(n,k_n)\ge u_\beta\big)=\big[1-\P\big(t(n,k_n)<u_\beta\big)\big]^{\ell_n}.
$$
Since $d(n,1-\beta)$ grows logarithmically in $n$, Theorem~\ref{thm:mdfpp} applies and gives that
$$
\P\big(t(n,k_n)<u_\beta\big)=\big(1+o(1)\big)\big[1-\Phi\big(d(n,1-\beta)\big)\big].
$$
By the tail behaviour of the Gaussian distribution, in~\eqref{eq:normal2}, we obtain
$$
\P\big(t(n,k_n)<u_\beta\big)=\big(1+o(1)\big)\frac{1}{\sqrt{2\pi}\,d(n,1-\beta)}e^{-d(n,1-\beta)^2/2}=\big(1+o(1)\big)\frac{\log(1/(1-\beta))}{\ell_n}.
$$
We conclude that
$$
\P\big(\tau^\star(n,k_n)\ge u_\beta\big)=\bigg[1-\big(1+o(1)\big)\frac{\log(1/(1-\beta))}{\ell_n}\bigg]^{\ell_n}\to1-\beta
$$
as $n\to\infty$, as required.
 \end{proof}

Next we relate the quantiles of $\tau^\star(n,k_n)$ with $u_\beta$.

\begin{lemma}\label{lma:quantiles}
Suppose that $k_n=o(n^{1/22})$. Fix $\beta\in(0,1)$ and $\eps>0$ so that $0<\beta-\eps<\beta+\eps<1$. Then, for any $\beta$-quantile $q_\beta$ of $\tau^\star(n,k_n)$, for large $n$ we have
$$
u_{\beta-\eps}<q_\beta<u_{\beta+\eps}.
$$
\end{lemma}

\begin{proof}
    Fix $\beta\in(0,1)$ and $\eps>0$ so that $0<\beta-\eps<\beta+\eps<1$. By Lemma~\ref{lma:nottrivial} we have for all large $n$ that
    $$
    \P\big(\tau^\star(n,k_n)\le u_{\beta-\eps}\big)\le\P\big(\tau^\star(n,k_n)<u_{\beta-\eps/2}\big)\le\beta-\eps/4.
    $$
    Consequently, for large values of $n$ we have that $u_{\beta-\eps}$ is too small to be a $\beta$-quantile of $\tau^\star(n,k_n)$, and hence that $u_{\beta-\eps}<q_\beta$. Similarly, again by Lemma~\ref{lma:nottrivial}, we have for large $n$ that
    $$
    \P\big(\tau^\star(n,k_n)\ge u_{\beta+\eps}\big)\ge \beta+\eps/2,
    $$
    and hence that $q_\beta< u_{\beta+\eps}$.
\end{proof}

Our final lemma is a uniform bound on the probability that a the left-right crossing time of a rectangle is contained in a bounded interval.

\begin{lemma}\label{lma:interval}
    Suppose that $k_n=o(n^{1/22})$. For every $\beta\in(0,1)$ and $c>0$ we have
    $$
    \sup_{x\le u_{\beta}}\P\Big(t(n,k_n)\in[x-c,x)\Big)=o\Big(\frac{1}{n^{23/22}}\Big).
    $$
\end{lemma}

\begin{proof}
Let $v=v(n):=\E[t(n,k_n)]-\sqrt{4\log n}\cdot\sqrt{\Var(t(n,k_n))}$. By Theorem~\ref{thm:mdfpp} and~\eqref{eq:normal2} we have that
$$
\P\big(t(n,k_n)<v\big)=\big(1+o(1)\big)\big[1-\Phi(\sqrt{4\log n})\big]=O\Big(\frac{1}{n^2}\Big).
$$
Hence, it remains to show that
$$
\sup_{x\in[v,u_{\beta}]}\P\big(t(n,k_n)\in[t-c,t)\Big)=o\Big(\frac{1}{n^{23/22}}\Big).
$$

We first rewrite
$$
\P\big(t(n,k_n)\in[x-c,x)\big)=\P\big(t(n,k_n)<x\big)-\P\big(t(n,k_n)<x-c\big).
$$
Next we introduce a scaling function $h$ as
\begin{equation}\label{eq:h}
h(x)=(x-\E[t(n,k_n)])/\sqrt{\Var(t(n,k_n))},
\end{equation}
and note that $h(x)$ is negative on $[v,u_\beta]$. Applying Theorem~\ref{thm:mdfpp} (twice) with $\alpha=1/22$ shows that the above expression equals
$$
\Big(1+o\Big(\frac{1}{n^{1/11}}\Big)\Big)\big[1-\Phi\big(-h(x)\big)\big]-\Big(1+o\Big(\frac{1}{n^{1/11}}\Big)\Big)\big[1-\Phi\big(-h(x-c)\big)\big]
$$
Rearranging the terms above gives
$$
\frac{1}{\sqrt{2\pi}}\int_{-h(x)}^{-h(x-c)}e^{-y^2/2}\,dy+o\Big(\frac{1}{n^{1/11}}\Big)\big[1-\Phi(-h(x))\big].
$$
The above expression is increasing in $x$, and maximal over the given interval for $x=u_\beta$. This gives the further upper bound
$$
\frac{c}{\sqrt{\Var(t(n,k_n))}}e^{-d(n,1-\beta)^2/2}+o\Big(\frac{1}{n^{1/11}}\Big)\big[1-\Phi\big(d(n,1-\beta)\big)\big].
$$
Since $\ell\sim n/k_n$ and $\Var(t(n,k_n))\ge n/k_n$, by Lemma~\ref{lma:CD_variance}, we obtain by definition of $d$ and~\eqref{eq:normal2} that
$$
\sup_{x\in[v,u_{\beta}]}\P\big(t(n,k_n)\in[x-c,x)\big)=O\Big(\frac{\sqrt{\log n}}{n^{30/22}}\Big)+o\Big(\frac{1}{n^{23/22}}\Big),
$$
as required.
\end{proof}

 We now finally have the tools to prove Theorem~\ref{thm:preliminary}.
 
\begin{proof}[Proof of Theorem \ref{thm:preliminary}]
Given $\eps>0$ we have, by Lemmas~\ref{lma:nottrivial} and~\ref{lma:quantiles}, that for large $n$
$$
\beta-2\eps\le\P\big(\tau^\star(n,k_n)<u_{\beta-\eps}\big)\le\P\big(\tau^\star(n,k_n)<q_\beta\big)\le\P\big(\tau^\star(n,k_n)\le q_\beta\big)\le\P\big(\tau^\star(n,k_n)<u_{\beta+\eps}\big)\le\beta+2\eps.
$$
Since $\eps>0$ was arbitrary, we conclude that
$$
\lim_{n\to\infty}\P\big(\tau^\star(n,k_n)< q_\beta\big)=\lim_{n\to\infty}\P\big(\tau^\star(n,k_n)\le q_\beta\big)=\beta.
$$

In order to prove that the function (or, more precisely, sequence of functions) $f_n^\star=\ind_{\{\tau^\star(n,k_n)> q_\beta\}}$ is noise sensitive, we aim to bound the influences of the individual edges, to show that the sum of influences squared tends to zero as $n\to\infty$. The conclusion will then follow from the BKS Theorem.

First note that since the $\ell_n$ rectangles are disjoint, each edge is contained in at most one rectangle. Moreover, changing the value of an edge may affect the crossing time of the rectangle it is contained in, but not the crossing times of the remaining rectangles. In particular, edges not contained in any rectangle have influence zero.
Since all rectangles are of equal dimensions, it will suffice to bound the influence of an edge contained in the first rectangle $[0,n]\times[0,k_n-1]$. So, fix an edge $e$ in this rectangle.

To estimate the influence of $e$, note that since being pivotal (meaning that changing the value of $\omega_e$ will change the outcome of the function) does not depend on the value of the edge itself, we have
$$
\Inf_e(f_n^\star)=2\,\P(e\text{ pivotal})\,\P(\omega_e=a)=2\,\P(e\text{ pivotal},\omega_e=a).
$$
Next, note that if there exists a left-right distance-minimising path of the rectangle {\em not} containing $e$, then increasing the weight at $e$ has no effect on $t_1(n,k_n)$. However, if every left-right distance-minimising path of the rectangle contains $e$, then increasing $\omega_e$ from $a$ to $b$ will change $t_0(n,k_n)$ by an amount at most $b-a$. Hence, on the event that $e$ is pivotal and $\omega_e=a$, we have that $t_0(n,k_n)\in[q_\beta-(b-a),q_\beta)$, while the remaining rectangles all have crossing time at least $q_\beta$.
It follows, in particular, that
$$
\Inf_e(f_n^\star)\le2\,\P\big(t(n,k_n)\in[q_\beta-(b-a),q_\beta)\big).
$$
Next, fix $\eps>0$ such that $\beta+\eps<1$. By Lemma~\ref{lma:quantiles}, we have $q_\beta<u_{\beta+\eps}$ for large $n$. Consequently, it follows from Lemma~\ref{lma:interval} that
\begin{equation}\label{eq:inf_indiv}
\Inf_e(f_n^\star)=o\Big(\frac{1}{n^{23/22}}\Big).
\end{equation}


The $\ell_n$ rectangles are all contained in the square $[0,n]\times[0,n-1]$. Since the square consists of $O(n^2)$ edges, there exists a constant $C>0$ such that

$$\sum_{e\in E}\Inf_e(f_n^\star)^2\le Cn^2\max_{e\in E}\Inf_e(f_n^\star)^2.$$
Together with~\eqref{eq:inf_indiv} we get that
$$
\sum_{e\in E}\Inf_e(f_n^\star)^2= n^2\cdot o\Big(\frac{1}{n^{23/11}}\Big)=o\Big(\frac{1}{n^{1/11}}\Big).
$$
The desired conclusion now follows from the BKS Theorem.
\end{proof}

\section{Proof of main results}\label{sec:main}

We won't be able to derive an as precise description of the asymptotics for $\beta$-quantiles of $\tau(n,k)$ as for those of $\tau^\star(n,k)$. Nevertheless, having done much of the ground work in the previous section, we will be able to finish up the proof of Theorem~\ref{thm:main} without much effort.

\begin{proof}[Proof of Theorem~\ref{thm:main}]
    By definition of a quantile we have for any $\beta$-quantile $q_\beta$ of $\tau(n,k_n)$ that
    $$
    \P\big(\tau(n,k_n)\le q_\beta\big)\ge\beta\quad\text{and}\quad\P\big(\tau(n,k_n)< q_\beta\big)\le\beta.
    $$
    Since by definition we have $\tau(n,k_n)\le \tau^\star(n,k_n)$, it follows that for every $\beta$-quantile $q_\beta$ of $\tau(n,k_n)$ there exists a $\beta$-quantile $q_\beta^\star$ of $\tau^\star(n,k_n)$ such that $q_\beta\le q_\beta^\star$.
    Fix $\eps>0$ so that $\beta+\eps<1$. Then, $q_\beta<u_{\beta+\eps}$ for large $n$ by Lemma~\ref{lma:quantiles}. It thus follows, for every $c>0$, that
    $$
    \beta\le\P\big(\tau(n,k_n)\le q_\beta\big)\le\P\big(\tau(n,k_n)<q_\beta\big)+\sup_{x\le u_{\beta+\eps}}\P\big(\tau(n,k_n)\in[x-c,x)\big).
    $$
    The definition of a quantile, the union bound, and Lemma~\ref{lma:interval} give the upper bound
    $$
    \beta+n\cdot\sup_{x\le u_{\beta+\eps}}\P\big(t(n,k_n)\in[x-c,x)\big)=\beta+o\Big(\frac{1}{n^{1/22}}\Big).
    $$
    This proves the first part of the theorem.

    In order to prove the second part of the second part we aim once again to bound the individual influences. Let $e$ be an edge, and recall that
    $$
    \Inf_e(f_n)=2\,\P(e\text{ pivotal})\,\P(\omega_e=a)=2\,\P(e\text{ pivotal},\omega_e=a).
    $$
    Each edge in the square $[0,n]\times[0,n-1]$ is contained in at most $k_n$ translates of the rectangle $[0,n]\times[0,k_n-1]$. Changing the value of the edge may affect the crossing time of the rectangles it is contained in, but not the crossing times of the remaining rectangles. More precisely, increasing the weight at $e$ from $a$ to $b$ will affect $t_i(n,k_n)$ if and only if every left-right distance minimising path of the rectangle $(0,i)+[0,n]\times[0,k_n-1]$ contains $e$. In this case, the change can result in an increase of at most $b-a$. It follows by the union bound that
    $$
    \Inf_e(f_n)\le2k_n\,\P\big(t(n,k_n)\in[q_\beta-(b-a),q_\beta)\big),
    $$
    which by Lemma~\ref{lma:interval} gives
    $$
    \max_{e\in E}\Inf_e(f_n)=o(1/n).
    $$
    Consequently,
    $$
    \sum_{e\in E}\Inf_e(f_n)^2\le Cn^2\max_{e\in E}\Inf_e(f_n)^2=o(1).
    $$
    Thus, the conclusion of the theorem follows from the BKS Theorem.
\end{proof}

Although not necessary, let us also provide a rough estimate on the quantiles of $\tau(n,k_n)$. For $\beta\in(0,1)$ let $u_\beta$ be defined as in~\eqref{eq:t_beta} and set
$$
\bar u_\beta:=\E[t(n,k_n)]-\sqrt{\Var(t(n,k_n))}\sqrt{2\log n-\log(2\log n)-2\log\big(\sqrt{2\pi}\beta\big)}.
$$
We claim that for fixed $\beta\in(0,1)$ and $\eps>0$ so that $0<\beta-\eps<\beta+\eps<1$, any $\beta$-quantile $q_\beta$ of $\tau(n,k_n)$ satisfies for large $n$ that
\begin{equation}\label{eq:quantile}
\bar u_{\beta-\eps}<q_\beta<u_{\beta+\eps}.
\end{equation}

Recall that, by construction, we have $\tau(n,k_n)\le \tau^\star(n,k_n)$. So that for every $\beta$-quantile $q_\beta$ there exists a $\beta$-quantile $q^\star_\beta$ of $\tau^\star(n,k_n)$ such that $q_\beta\le q_\beta^\star$. The upper bound in~\eqref{eq:quantile} is thus immediate from Lemma~\ref{lma:quantiles}.


    For the lower bound, let $N_\beta$ denote the number of rectangles with crossing time less than $\bar u_\beta$, i.e.\ let $N_\beta:=\#\{i=1,2,\ldots,n-k_n:t_{i-1}(n,k_n)<\bar u_\beta\}$. Then, using Markov's inequality,
    $$
    \P\big(\tau(n,k_n)<\bar u_\beta\big)=\P(N_\beta\ge1)\le n\,\P\big(t(n,k_n)<\bar u_\beta\big).
    $$
    Theorem~\ref{thm:mdfpp} and~\eqref{eq:normal2} gives that
    $$
    \P\big(t(n,k_n)<\bar u_\beta\big)=\big(1+o(1)\big)\frac{\beta}{n}.
    $$
    Hence we obtain for large $n$ that
    $$
    \P\big(\tau(n,k_n)\le \bar u_{\beta-\eps}\big)\le\P\big(\tau(n,k_n)<\bar u_{\beta-\eps/2}\big)\le\beta-\eps/4.
    $$
    This shows that $\bar u_{\beta-\eps}$ is too small to be a $\beta$-quantile for $\tau(n,k_n)$ when $n$ is large, and hence proves the lower bound in~\eqref{eq:quantile}.

\begin{proof}[Proof of Theorem~\ref{thm:fluctuations}]
We begin with the observation that
$$
\P\big(\tau(n,k)\in[x-c,x]\big)=\P\bigg(\bigcup_{i=0}^{n-k-1}\big\{t_i(n,k)\in[x-c,x]\big\}\cap\big\{t_j(n,k)\ge x-c,\forall j\neq i\big\}\bigg).
$$
For each $i$ we may find $\lfloor n/k\rfloor-1$ indices $j$ for which the rectangles corresponding to the variables $t_j(n,k)$ are disjoint, and disjoint of the rectangle corresponding to $t_i(n,k)$. The corresponding crossing times are thus independent, and exercising the union bound, we obtain that
\begin{equation}\label{eq:superb}
\P\big(\tau(n,k)\in[x-c,x]\big)\le n\,\P\big(t(n,k)\in[x-c,x]\big)\,\P\big(t(n,k)\ge x-c\big)^{\lfloor n/k\rfloor-1}.
\end{equation}
We shall bound both probabilities in the above right-hand side using Theorem~\ref{thm:mdfpp}.

Fix $\alpha<1/22$ and set $\beta=1-1/22$ so that $\beta<1-\alpha$. Let
$$
y:=\E[t(n,k_n)]-\sqrt{\Var(t(n,k_n))}\cdot\sqrt{2\beta\log n}.
$$
Then, Theorem~\ref{thm:mdfpp} and~\eqref{eq:normal2} give
\begin{equation}\label{eq:beta_bound}
\P\big(t(n,k_n)\ge y-c\big)= 1-(1+o(1))\frac{1}{\sqrt{4\pi\beta\log n}\cdot n^{\beta}},
\end{equation}
and hence that
\begin{equation}\label{eq:superpoly}
\P\big(t(n,k_n)\ge y-c\big)^{\lfloor n/k_n\rfloor-1}\le \exp\Big(-n^{1-\alpha-\beta}/\sqrt{4\pi\beta\log n}\Big),
\end{equation}
which decays faster than any polynomial since $\beta<1-\alpha$.

The reminder of the proof will closely follow that of Lemma~\ref{lma:interval}. Let again $h(x)$ be defined as in~\eqref{eq:h}. Then, by an analogous calculation as that leading to~\eqref{eq:beta_bound}, we obtain that
\begin{equation}\label{eq:supereroi}
\P\big(t(n,k_n)<h^{-1}(\sqrt{4\log n})\big)=O\Big(\frac{1}{n^2}\Big).
\end{equation}
A calculation analogous to those in Lemma~\ref{lma:interval} gives that for $h^{-1}(\sqrt{4\log n})\le x\le y$ we have
$$
\P\big(t(n,k_n)\in[x-c,x]\big)=\frac{1}{\sqrt{2\pi}}\int_{-h(x)}^{-h(x-c)}e^{-z^2/2}\,dz+o\Big(\frac{1}{n^{1/11}}\Big)\big[1-\Phi(-h(x))\big],
$$
which is maximal for $x=y$. Together with~\eqref{eq:supereroi} we thus get, for some constant $C<\infty$, that
\begin{equation}\label{eq:supernormal}
\sup_{x\le y}\P\big(t(n,k_n)\in[x-c,x]\big)\le \frac{C}{n^{1/11+\beta}\sqrt{\log n}}=\frac{C}{n^{1+1/22}\sqrt{\log n}}.
\end{equation}
Finally, combining~\eqref{eq:superb},~\eqref{eq:superpoly} and~\eqref{eq:supernormal}, considering the cases $x\le y$ and $x>y$ separately, we obtain that
$$
\sup_{x\ge0}\P\big(\tau(n,k)\in[x-c,x]\big)\le\frac{C}{n^{1/22}\sqrt{\log n}},
$$
as required.
\end{proof}

\section{Further directions}\label{sec:conjectures}

We will devote this last section to indicate some future directions of research and open problems related to Benjamini's problem and the work of this paper.

We started out with the problem of whether `being above the median' is a noise sensitive property for the point-to-point passage time $T_n=T(0,n{\bf e}_1)$. Due to the limited understanding of fluctuations in first-passage percolation, we have had to resort to restricting the problem in order to make progress. This led us, in the introduction, to call for
$$
\sup_{x\ge0}\P\big(T_n\in[x,x+c]\big)\to0\quad\text{as }n\to\infty,
$$
for every $c>0$.

More precise results regarding the nature of fluctuations have been established in related models of spatial growth, such as increasing subsequences in the plane, last-passage percolation with exponential or geometric weights, Brownian last-passage percolation, as well as for the largest eigenvalue of random matrices. It appears as if these results are in themselves insufficient to answer Benjamini's question. In addition, these settings do not fit into the framework of Boolean functions. Hence, solving Benjamini's problem in these settings remains an interesting open problem.

Another relevant question regards the relation between noise sensitivity of being above a certain quantile of some real-valued sequence of functions $f_n:\{0,1\}^n\to\R$, and the asymptotic independence of $f_n(\omega)$ and $f_n(\omega^\eps)$. In particular, given that
\begin{equation}\label{eq:nsm}
\P\big(f_n(\omega)>q,f_n(\omega^\eps)>q\big)-\P\big(f_n(\omega^\eps)>q\big)^2\to0\quad\text{as }n\to\infty
\end{equation}
for every quantile $q$ of $f_n$, is it then also true that
\begin{equation}\label{eq:nsc}
\textup{Corr}\big(f_n(\omega),f_n(\omega^\eps)\big)\to0\quad\text{as }n\to\infty?
\end{equation}
For many sequences it is natural to expect that the mean of $f_n$ corresponds to one of its quantiles, and thus that if~\eqref{eq:nsm} holds, then the signs of $f_n(\omega)-\E[f_n]$ and $f_n(\omega^\eps)-\E[f_n]$ are asymptotically independent, and hence that~\eqref{eq:nsc} should hold. We do not know whether this is true in general.

Finally, the reader may wonder why we consider the restricted square crossing time $\tau(n,k)$ now that already the rectangle crossing time $t(n,k)$ is known to obey a Gaussian central limit theorem. Well, for fixed $k$ we expect that $t(n,k)$ being above its median is a {\em noise stable} property, and hence not noise sensitive. Indeed, the case $k=1$ coincides with the classical Majority function on $n$ bits, which is well-known to be noise stable; see e.g.~\cite{garban_steif_book}. For diverging sequences $(k_n)_{n\ge1}$ we conjecture that `being above the median' is a noise sensitive property for $t(n,k_n)$. We motivate this by an heuristic calculation similar to~\eqref{eq:inf_firstbound}, which suggests that for a given edge $e$
$$
\Inf_e(\ind_{\{t(n,k)>m\}})\,\asymp\, \P\big(e\in\pi_n\big)\,\P\big(|t(n,k)-m|\le b-a\big)\,\asymp\, \frac{1}{k}\frac{1}{\sqrt{\Var(t(n,k))}},
$$
and hence that (there are about $nk$ influential edges)
$$
\sum_{e}\Inf_e(\ind_{\{t(n,k)>m\}})^2\asymp \frac{n}{k}\frac{1}{\Var(t(n,k))}.
$$
It is believed that $\Var(t(n,k))\asymp n/\sqrt{k}$ whenever $k=o(n^{2/3})$, and this has been proved to be the case in a related model by Dey, Joseph and Peled~\cite{DJP}. However, in first-passage percolation, the best bounds only give $\Var(t(n,k))\ge n/k$, which would give a constant upper bound on the sum of influences squared; see~\cite{Chatterjee2009CentralLT}. Hence, one would need to improve upon the variance bound in order to establish noise sensitivity of the rectangle crossing variables.

\appendix

\section{On the empirical distribution of cylinder crossing times}

In this section we revisit an argument due to van den Berg and Kesten~\cite{vdBK93}, regarding the empirical distribution of edges along a geodesic, which is required to complete the proof of Lemma~\ref{lma:CD_variance} in the main text. Their argument has been revisited in the works of Marchand~\cite{M02}, Gorski~\cite{gorski24} and Cid~\cite{cid23}. We shall here show how the argument extends from point-to-point travel times in the plane to cylinder crossing times. The presentation will follow that of~\cite{cid23} closely, which in turn was inspired by~\cite{gorski24}. Although the argument extends to dimensions $d\ge2$, here we give the argument in the planar setting.

In this section we consider the $\Z^2$ nearest neighbour lattice equipped with i.i.d.\ weights $(\omega_e)$ with common distribution $F$ supported on $[0,\infty)$. Let $r:=\inf\{x\ge0:F(x)>0\}$ denote the infimum of the support of $F$. We shall here consider any weight distribution $F$ satisfying
\begin{equation}\label{eq:useful}
F(0)<p_c\quad\text{and}\quad F(r)<\vec{p}_c,
\end{equation}
where $p_c=1/2$ denotes the critical probability for percolation and $\vec{p}_c>1/2$ the critical probability for oriented percolation on $\Z^2$. As in~\eqref{eq:t_rect}, we let $t(n,k)$ denote the left-right crossing time of the rectangle $[0,n]\times[0,k-1]$, with respect to the weight configuration $\omega=(\omega_e)$, and denote by $\pi(n,k)$ the set of edges belonging to some geodesic for $t(n,k)$.

\begin{proposition}\label{prop:empirical}
    For any weight distribution $F$ satisfying~\eqref{eq:useful} and any interval $I=[\alpha,\beta]$ such that $\P(\omega_e\in I)=F(\beta)-F(\alpha-)>0$ there exists a constant $c=c(I)>0$ such that for all $n\ge k\ge1$
    $$
    \E\big[\#\{e\in\pi(n,k):\omega_e\in I\}\big]\ge cn.
    $$
\end{proposition}

As mentioned, the analogous result for $T(0,n \boldsymbol{e}_1)$ is a well-known result due to van den Berg and Kesten~\cite[Equation (2.16)]{vdBK93}. The first condition in~\eqref{eq:useful} is necessary, since the optimal path otherwise is known to contain $O(\log n)$ edges with positive weight. The second condition is believed to be superfluous, which in~\cite{M02} was shown to be the case for $T(0,n{\bf e}_1)$.

The approach involves two main steps: The first is a coarse-graining step, and the second involves a resampling argument. The first step is identical to the full-plane setting, so we only sketch the proof of Lemma~\ref{lma:grains}. The second step requires some minor modifications, so the proof of Lemma~\ref{lma:resampling} and the conclusion of the proof are given in detail.

Let $\|\cdot\|$ denote $\ell^1$-distance on $\R^2$. For the coarse-graining step, given an integer $L\ge1$, we tile $\Z^2$ by (closed) $\ell_1$-balls $B(z,L):=\{x:\|x-z\|\le L\}$, centred at the points of the form $(i2L+jL,jL)$ for $i,j\in\Z$. We denote by $V_L$ this set of vertices.
For $m\ge2$ and $z\in V_L$, we let $A_{L,m}(z)$ denote the event that any path connecting $B(z,L)$ to $\partial B(z,mL)$ picks up weight at least $(m-1)L(r+\delta)$. Under the condition~\eqref{eq:useful} it is well-known (see e.g.~\cite[Lemma~5.5]{vdBK93}) that there exists $\delta>0$ such that
\begin{equation}\label{eq:grains}
\P\big(A_{L,m}(z)\big)\to1\quad\text{as }L\to\infty.
\end{equation}
For the remainder of this section, we fix $\delta>0$ accordingly.

\begin{lemma}\label{lma:grains}
For every $m\ge2$ and all $L\ge1$ sufficiently large, there exist constants $\theta,\gamma>0$ and $C<\infty$ such that for all $x,y\in\Z^2$ we have
$$
\P\big(\exists\text{ path $\pi:x\leftrightarrow y$ visiting at most $\theta\|y-x\|$ tiles for which $A_{L,m}(z)$ occurs}\big)<C\exp\big(-\gamma\|y-x\|\big).
$$
\end{lemma}

\begin{proof}[Sketch of proof]
The statement follows from a standard Peierls argument; see~\cite[Lemma~5.2]{vdBK93} for details. Each path $\pi$ between $x$ and $y$ visits at least $\frac{1}{4L}\|y-x\|$ distinct tiles $B(z,L)$ and, considering a subsequence, we can find $\frac{1}{4mL}\|y-x\|$ tiles $B(z,L)$ for which the larger regions $B(z,mL)$ are disjoint, and so that the events $A_{L,m}(z)$ are independent. For a proportion $\P(A_{L,m}(z))$ of these the event $A_{L,m}(z)$ occurs, and this proportion is by~\eqref{eq:grains} close to 1 if $L$ is large. A large deviation estimate of the binomial distribution gives that it is exponentially unlikely (in $\|y-x\|$) that $\pi$ picks up fewer than $\frac{1}{8mL}\|y-x\|$ tiles $B(z,L)$ for which the event $A_{L,m}(z)$ occurs.
\end{proof}

The next lemma is proved via a resampling argument. It will require that $m\ge2$ is fixed so that
\begin{equation}\label{eq:m_cond}
    \beta+(m+3)\Big(r+\frac\delta2\Big)<(m-1)(r+\delta).
\end{equation}
For ease of notation, let $D_{L,m,k}(z)$ denote the event that both endpoints of $\pi(n,k)$ lie outside $B(z,mL)$ and that $\pi(n,k)$ visits some vertex in $B(z,L)$.

\begin{lemma}\label{lma:resampling}
    For $m$ satisfying~\eqref{eq:m_cond} and all $L\ge1$ there exists $c=c(L,m)$ such that for all $z\in V_L$ and $k\ge 2L$ we have
    \begin{align*}
    &\P\big(\exists e\in\pi(n,k)\cap B(z,mL)\text{ such that }\omega_e\in I\big)\ge c\,\P\big(D_{L,m,k}(z)\cap A_{L,m}(z)\big).
    \end{align*}
\end{lemma}

\begin{proof}
    Fix $z\in V_L$ such that $B(z,L)$ intersects the rectangle $[0,n]\times[0,k-1]$. Let $\omega,\omega'$ be two independent weight configurations on the $\Z^2$ lattice, both with marginal distribution $F$. We define a new weight configuration $\omega^\ast=(\omega_e^\ast)$ where
    \begin{equation}
        \omega_e^\ast:=\left\{
        \begin{aligned}
            \omega'_e & & \text{for }e\in B(z,mL)\setminus B(z,L)\\
            \omega_e & & \text{otherwise}.
        \end{aligned}
        \right.
    \end{equation}
    We may think of $\omega^\ast$ as the configuration obtained after resampling the original weight configuration $\omega$ on $B(z,mL)\setminus B(z,L)$. We let $G_{L,m}(z)$ denote the event that $\omega'_e\in I$ for every edge $e$ with one endpoint in the interior of $B(z,mL)$ and one on the boundary, and $\omega'_e<(r+\delta/2)$ for every other edge in $B(z,mL)\setminus B(z,L)$.

    Note that on the event $D_{L,m,k}(z)$ the path $\pi(n,k)$ crosses $\partial B(z,mL)$, visits the tile $B(z,L)$, and then crosses $\partial B(z,mL)$ again. Let $x$ and $y$ denote the first and last points on $\partial B(z,mL)$ visited by $\pi(n,k)$, when oriented from left to right. Let $\gamma$ be any path from $x$ to $y$ contained inside $B(z,mL)\setminus B(z,L)$ of minimal length. Note that it is possible that $\gamma$ is not fully contained in the rectangle $[0,n]\times[0,k-1]$. In this case we replace $\gamma$ by a path $\gamma'$ by circling around $B(z,L)$, adding up to $8L$ additional edges. (Circling around $B(z,L)$ is possible inside the rectangle since $k\ge 2L$, by assumption.) Note that both $\gamma$ and $\gamma'$ can be chosen to have length at most $2(m-1)L+8L$.
    
    On the event $D_{L,m,k}(z)\cap A_{L,m}(z)$, the segment of $\pi(n,k)$ between $x$ and $y$ picks up weight at least $2(m-1)L(r+\delta)$, whereas on the event $G_{L,m}(z)$, the paths $\gamma$ and $\gamma'$ pick up weight at most $2\beta+\big(2(m-1)L+8L\big)(r+\delta/2)$, which by assumption~\eqref{eq:m_cond} is strictly smaller. That is, if $G_{L,m}(z)$ occurs after `resampling' the weight configuration on $B(z,mL)\setminus B(z,L)$, the paths $\gamma$ and $\gamma'$ are at least as `fast' as the corresponding segment of $\pi(n,k)$ was before resampling. Since no edges outside $B(z,mL)\setminus B(z,L)$ are resampled, it follows that after resampling there also exists a path that visits the interior of $B(z,mL)$ that is `faster' than any path that does not. Consequently, the geodesic will continue to visit the interior of $B(z,mL)$, and hence pick up an edge inside $B(z,mL)$ with weight in $I$.

    More formally, the above discussion amounts to the statement
    $$
    \{\omega\in D_{L,m,k}(z)\cap A_{L,m}(z)\}\cap\{\omega'\in G_{L,m}(z)\}\subseteq\big\{\omega^\ast\in\{\exists e\in\pi(n,k)\cap B(z,mL)\text{ with weight in }I\}\big\}.
    $$
    By independence of $\omega$ and $\omega'$, and since $\omega$ and $\omega^\ast$ are equal in distribution, we have that
    \begin{align*}
    \P\big(\exists e\in\pi(n,k)\cap B(z,mL)\text{ with weight in }I\big)&\ge\P\big(\omega\in D_{L,m,k}(z)\cap A_{L,m}(z),\omega'\in G_{L,m}(z)\big)\\
    &=\P\big(D_{L,m,k}(z)\cap A_{L,m}(z)\big)\P\big(G_{L,m}(z)\big),
    \end{align*}
    as required.
\end{proof}

We are now ready to complete the proof of the proposition.

\begin{proof}[Proof of Proposition~\ref{prop:empirical}]
    Fix $m$ so that~\eqref{eq:m_cond} holds and $L$ so that the conclusion of Lemma~\ref{lma:grains} is valid. Note that it suffices to prove the proposition for large $n$, as for small $n$ we can always lower bound the expectation by $n$ times the probability that all edges in the rectangle $[0,n]\times[0,k-1]$ take values in $I$. Moreover, it will suffice to prove the proposition for $k\ge2L$, since for $k$ bounded the number of weights in $I$ picked up by the geodesic can be lower bounded by the number of vertical lines where all edges reaching beyond the line have weight in $I$. We hence take $n\ge2mL$ and $n\ge k\ge2L$.

    Since each edge belongs to $B(z,mL)$ for $m^2$ different $z\in V_L$, we have that
    $$
    \E\big[\#\{e\in\pi(n,k):\omega_e\in I\}\big]=\frac{1}{m^2}\E\bigg[\sum_{z\in V_L}\sum_{e\in B(z,mL)}\ind_{\{e\in\pi(n,k)\text{ and }\omega_e\in I\}}\bigg],
    $$
    which by Lemma~\ref{lma:resampling} is lower bounded by
    $$
    \frac{1}{m^2}\sum_{z\in V_L}\P\big(\exists e\in\pi(n,k)\cap B(z,mL)\text{ such that }\omega_e\in I\big)\ge\frac{1}{m^2}\sum_{z\in V_L}c\,\P\big(D_{L,m,k}(z)\cap A_{L,m}(z)\big).
    $$
    The fact that each endpoint of $\pi(n,k)$ can be in $B(z,mL)$ for at most $m^2$ different $z\in V_L$ gives the further lower bound
    $$
    \frac{c}{m^2}\Big(\E\big[\#\{z\in V_L:\pi(n,k)\text{ visits $B(z,L)$ and $A_{L,m}(z)$ occurs}\}\big]-2m^2\Big).
    $$
    Finally, since there are $k$ different starting and endpoints for $\pi(n,k)$, and $k\le n$, we obtain from Lemma~\ref{lma:grains} that
    $$
    \E\big[\#\{e\in\pi(n,k):\omega_e\in I\}\big]\ge\frac{c}{m^2}\Big(\theta n\big(1-n^2C\exp(-\gamma n)\big)-2m^2\Big),
    $$
    which for large $n$ gives a lower bound of the form $c'n$, for a constant $c'=c'(L,m)$ not depending on $k$ and $n$, as required.
\end{proof}

\printbibliography

\end{document}